\documentclass[11pt]{article}
\usepackage{amsmath,amssymb,amsthm}
\usepackage{latexsym}
\usepackage[dvips]{graphicx}

\newtheorem{definition}{Definition}
\newtheorem{theorem}{Theorem}

\newtheorem{example}{Example}
\newtheorem{assumption}{Assumption}
\newtheorem{lemma}{Lemma}
\newtheorem{remark}{Remark}

\usepackage{color}

\def\Rbb{\mathbb{R}}

\def\<{\langle}
\def\>{\rangle}

\title{Affine Invariant Divergences associated with Composite Scores and its Applications}
\author{
  Takafumi Kanamori\\ Nagoya University \\ \tt{kanamori@is.nagoya-u.ac.jp}
  \and
  Hironori Fujisawa\\ The Institute of Statistical Mathematics\\ \tt{fujisawa@ism.ac.jp}
 }

\begin{document}
\maketitle

\begin{abstract}
In statistical analysis, measuring a score of predictive performance is an important task. 
In many scientific fields, appropriate scores were tailored to tackle the problems at
hand. A proper score is a popular tool to obtain statistically consistent forecasts. 
Furthermore, a mathematical characterization of the proper score was studied. 
As a result, it was revealed that the proper score corresponds to a Bregman divergence, 
which is an extension of the squared distance over the set of probability distributions. 
In the present paper, 
we introduce composite scores as an extension of the typical scores 
in order to obtain a wider class of probabilistic forecasting. 
Then, we propose a class of composite scores, named H\"{o}lder scores, 
that induce equivariant estimators. 
The equivariant estimators have a favorable property, implying that the estimator is 
transformed in a consistent way, when the data is transformed. 
In particular, we deal with the affine transformation of the data. 
By using the equivariant estimators under the affine transformation, 
one can obtain estimators that do no essentially depend on the choice of the system of
 units in the measurement. Conversely, we prove that the H\"{o}lder score is characterized by 
the invariance property under the affine transformations. 
Furthermore, we investigate statistical properties of the estimators using 
H\"{o}lder scores for the statistical problems including estimation of regression 
functions and robust parameter estimation, 
and illustrate the usefulness of the newly introduced scores for statistical forecasting. 

\end{abstract}
keywords: 
composite score; divergence; Bregman score; H\"{o}lder score; affine invariance.

\section{Introduction}\label{sec:Introduction}

In statistical analysis, an important task is to measure a score or a loss of the
prediction performance.  
In many fields in which probabilistic forecasting is required, 
appropriate scoring rules or loss functions are tailored to tackle the scientific problems at hand, 
e.g., weather and climate prediction \cite{bremnes04:_probab_forec_precip_terms_quant,brier50:_verif}, 
computational finance \cite{duffie97:_overv_value_risk}, and so forth. 

Under an uncertain situation, the prediction is described by using the probability
distribution. The probability distribution for the prediction is expected to put much 
weight to outcomes that are likely to materialize in the future. 
Hence, the score is formalized as a function taking two inputs, i.e., 
a probability distribution for the prediction and an outcome. 
In order to achieve high prediction performance on average, 
ideally, optimization of the expected score is conducted. 
When the identically and independently distributed (i.i.d.) samples are available, 
the expected score is approximated by the empirical mean over the samples. 
By optimizing the empirical mean of the score over a statistical model for the
prediction, one will obtain a probability distribution attaining high prediction
performance. 

The above statistical procedure is formalized as the statistical inference using the
scores or scoring rules
\cite{brier50:_verif,dawid07,gneiting07:_stric_proper_scorin_rules_predic_estim,hendrickson71:_proper_scores_probab_forec}. 
We regard the score as a loss to be minimized. 
The estimator obtained from the score is called the optimum score estimator. 
To obtain a good estimator, scores need to satisfy some assumptions. 
A typical assumption is that the score is proper. 
Given a probability distribution of outcomes, 
the optimal value of the expected proper 
score is attained by setting the prediction probability to be the true probability
distribution. Under mild assumptions, optimization of the proper score averaged over the
observed samples produces a statistically consistent estimator. 
The proper score is a special case of M-estimation \cite{huber64:_robus}, 
and the statistical property of the proper score have been studied in the framework of
M-estimators \cite[Chap.\,5]{vaart00:_asymp_statis}.

The proper score is a basic element that yields important concepts in statistical 
inference. According to \cite{dawid07}, the proper score introduces a divergence, 
which is a discrepancy measure between two probability distributions. 
The divergence is regarded as a generalization of the (squared) distance, and 
induces a sort of topological structure over the statistical model. 
As a result, geometrical structures such as the Riemannian metric and affine connection
are defined over the geometrical space consisting of probability distributions. 
Such a geometrical structure is closely related to the statistical properties of the estimator
using the corresponding proper score. 
Bregman divergence \cite{bregman67:_relax_method_of_findin_commonc}
is an important class of divergences, since it is closely related to
the proper score. 
A major milestone in the theoretical approach 
is the characterization of the proper score by using the Bregman divergence 
\cite{abernethy12:_charac_scorin_rules_linear_proper,%
gneiting07:_stric_proper_scorin_rules_predic_estim,hendrickson71:_proper_scores_probab_forec}. 
More precisely, any proper score produces a Bregman divergence, and 
a given Bregman divergence yields a proper score. 
The correspondence established a way to investigate the proper score 
by using the Bregman divergence on statistical models.

In the present paper, we introduce composite scores as an extension of the proper scores
in order to obtain a wider class of probabilistic forecasting. 
Then, we propose a class of composite scores, named H\"{o}lder scores, that induce
equivariant estimators \cite{berger85:_statis_decis_theor_bayes_analy}. 
The equivariant estimator is a class of estimators having a favorable property, 
implying that the estimator is transformed in a consistent way, when the data is transformed. 
In particular, we deal with the affine transformation of the data, i.e., 
$\omega\mapsto\sigma^{-1}(\omega-\mu)$ for the data $\omega\in\Rbb^d$, 
where $\sigma$ is a $d$ by $d$ invertible matrix and $\mu$ is a $d$-dimensional vector. 
The normalization of data is a typical example of affine transformations. 
Each element of the normalized data has zero sample-mean and unit sample-variance. 
Thus, for the normalized data, the statistical comparison of each component is
reasonable. 
As an example of the equivariant estimators under the affine transformation, 
let us consider the estimation of the mean value $\theta$ of a one-dimensional probability
distribution. 
When all samples are transformed from $\omega\in\Rbb$ into $\sigma^{-1}{}(\omega-\mu)$
with the constants $\mu\in\Rbb$ and $\sigma\neq0$, 
also the estimator $\widehat{\theta}$ of the mean value 
$\theta$ should be transformed into $\sigma^{-1}(\widehat{\theta}-\mu)$.  
By using the equivariant estimators under the affine transformation, 
the estimate does not essentially depend on the choice of the system of units in the
measurement. 
In addition, we show a characterization of the H\"{o}lder scores. 
Similarly to the correspondence between the proper scores and the Bregman divergences, 
the composite scores correspond to a class of divergences. 
When the divergence is invariant under the data transformation, 
the corresponding composite score provides an equivariant estimator. 
We prove that the H\"{o}lder score is characterized by the affine invariance of the
associated divergence, i.e., among a class of composite scores, 
only H\"{o}lder score provides the equivariant estimator under affine transformations. 
Furthermore, we investigate statistical properties of the estimators derived from 
H\"{o}lder scores for the statistical problems including estimation of regression 
functions and robust parameter estimation.

As pointed out in \cite{bremnes04:_probab_forec_precip_terms_quant}, scores of continuous
variables have so far received little attention. In this paper, our main concern is the
scores of continuous variables. 
The invariance under affine transformations is a specific property for continuous
variables.

The remainder of the article is organized as follows. 
In Section~\ref{sec:Proper_Scoring_Rules_and_Divergences}, 
we define composite scores and associated divergences. 
Bregman scores and its separable variant are also introduced as an important class of
composite scores. 
Then, we show a way to use composite scores to probabilistic forecasting. 
In Section~\ref{sec:Holder-score}, 
we define H\"{o}lder scores, and demonstrate the relation between H\"{o}lder scores and Bregman scores. 
In Section~\ref{sec:Affine_invariance_Holder-div}, 
we define the affine invariance of divergences, and show that 
the H\"{o}lder score induces the affine invariant divergences and equivariant estimators. 
Conversely, we prove that 
H\"{o}lder score is characterized by the affine invariance of the associated divergence. 
In Section~\ref{sec:Statistical_Inference_using_Holder_score}, 
the H\"{o}lder score is used to statistical problems including regression problems and
robust estimation. In particular, 
the robustness property of the H\"{o}lder score is presented. 
In Section~\ref{sec:conclusion}, we close this article with a discussion of the 
possibility of the newly introduced class of scores.

\section{Composite Scores and Associated Divergences}
\label{sec:Proper_Scoring_Rules_and_Divergences}

In this section, we define composite scores and associated divergences. 
Then, we introduce estimators using the composite scores. 

Let us summarize the notations to be used throughout the paper. 
Let $\Rbb$ be the set of all real numbers. 
The non-negative numbers are denoted as
$\Rbb_+=\{x\in\Rbb\,|\,x\geq0\}$. 
The interior set of a set $A$ is denoted as $A^\circ$. 
Thus, $\Rbb_+^\circ$ implies the set of all positive real numbers, i.e.,
$\Rbb_+^\circ=\{x\in\Rbb\,|\,x>0\}$. 
For a sample space $\Omega$, let $\mathcal{B}$ be a $\sigma$-algebra of subsets of $\Omega$
and $m:\mathcal{B}\rightarrow\Rbb_+$ 
be a $\sigma$-finite measure  on $(\Omega,\mathcal{B})$. 
The set of all measurable functions on $\Omega$ is denoted as $L_0$, i.e., 
$L_0=\{f:\Omega\rightarrow\Rbb\,|\,\text{$f$ is measurable on
$(\Omega,\mathcal{B},m)$}\}$. 
For $f\in{L_0}$, the integral $\int_{\Omega}f(\omega)dm(\omega)$ is denoted as $\<f\>$. 
Let $\|\cdot\|_\alpha$ for $1\leq \alpha<\infty$ be the $L_\alpha$-norm, i.e., 
$\|f\|_\alpha=\<|f|^\alpha\>^{1/\alpha}$, and 
$\|\cdot\|_{\infty}$ be the essential sup-norm. 
For $\alpha\geq1$, let $L_{\alpha}$ be
$L_\alpha=\left\{f\in{}L_0\,\big|\,\|f\|_\alpha<\infty\right\}$. 
For $\alpha=0$ or $\alpha\geq1$, $L_\alpha^+$ denotes the set of all non-negative and non-zero
functions in $L_\alpha$, i.e., 
$L_\alpha^+=\left\{ f\in{}L_\alpha\,|\,f\geq0, f\neq0 \right\}$. 
Provided a set of measurable and non-negative functions $\mathcal{F}\subset{}L_0^+$, 
$\mathcal{P}$ denotes the set of probability densities in $\mathcal{F}$, i.e., 
$\mathcal{P}=\{p\in\mathcal{F}\,|\,\<p\>=1\}$. 
For a differentiable function $\psi$, 
$\psi_i$ with the integer $i$ denotes the partial derivative of $\psi$ with respect to the
$i$-th argument, e.g., for $\psi(x,y)$, $\psi_1$ and $\psi_2$ denote
$\frac{\partial\psi}{\partial{x}}$ and $\frac{\partial\psi}{\partial{y}}$, respectively.

\subsection{Definitions}
\label{subsec:Scores_and_divergences}

Let us consider the probabilistic forecasts on a measurable space 
$(\Omega,\mathcal{B},m)$. 
Suppose that the probabilistic forecast is given by a probability density $q\in{L_1^+}$
satisfying $\<q\>=1$. 
For an outcome $\omega\in\Omega$, 
let $S_0(\omega,q)$ be a score of the forecast using $q$. 
When the probability density of the outcome is $p$, 
the expected score is given as 
\begin{align*}
 S_0(p,q):=\int_\Omega{}S_0(\omega,q)p(\omega)dm(\omega). 
\end{align*}
Suppose that the expected score satisfies the inequality 
$S_0(p,q)\geq{}S_0(p,p)$. 
Then, the minimization of the empirical mean of $S_0(\omega,q)$ over the statistical
model $q$ is expected to provide a good estimate of the probability density $p$. 
This approach is widely used in statistical inference. 
In this paper, the term score denotes the expected score $S_0(p,q)$, 
though typically the score denotes the function $S(\omega,g)$. 

Let us define a general form of scores. 
It is defined not only for probability densities but also non-negative functions. 
\begin{definition}[composite score]
 \label{def:scoring-rule}
Let $\mathcal{F}$ be a convex subset in $L_0^+$, and 
the set of probability densities in $\mathcal{F}$ is denoted as $\mathcal{P}$, i.e.,
$\mathcal{P}=\left\{p\in\mathcal{F}\,\big|\,\<p\>=1\right\}$. 
The function $S(f,g):\mathcal{F}\times\mathcal{F}\rightarrow\Rbb$ is called the composite
score on $\mathcal{F}$ if the following three conditions are satisfied: 
\begin{enumerate}
 \item $S(f,g)$ is of the form 
       \begin{align}
	\label{eqn:scoring-rule-expression}
	S(f,g)=T\left(\int_\Omega{}S_0(\omega,g) f(\omega)dm(\omega),\, g\right), 
       \end{align}
       where $S_0:\Omega\times\mathcal{F}\rightarrow{\Rbb}$ and
       $T:{\Rbb}\times\mathcal{F}\rightarrow{\Rbb}$. 
       The function $S_0(\cdot,g)f(\cdot)$ is assumed to be integrable for all
       $f,g\in\mathcal{F}$. 
 \item $S(f,g)\geq{}S(f,f)$ for all $f,g\in\mathcal{F}$. 
 \item For $p,q\in\mathcal{P}$, $S(p,q)=S(p,p)$ implies $p=q$ (almost surely). 
\end{enumerate}
\end{definition}
When the composite score $S(f,g)$ is defined only on the set of probability densities and 
the function $T$ is given as $T(c,g)=c$, 
the composite score is reduced to the expectation of a strictly proper score 
\cite{dawid07,%
gneiting07:_stric_proper_scorin_rules_predic_estim,%
gr03:_game_theor_maxim_entrop_minim,%
hendrickson71:_proper_scores_probab_forec}. 
Hence, the above definition is an extension of the strictly proper score. 
In Section \ref{sec:Holder-score}, we propose a class of composite scores 
with a non-trivial $T$.



\begin{remark}
 In our definition, the domain of the composite score is not necessarily a set of
 probability densities, but it can be a set of non-negative functions. 
 Likewise, in \cite{hendrickson71:_proper_scores_probab_forec}, the strictly proper scores are
 characterized on the set of non-negative functions. 
 The definition in the present paper simplifies mathematical analysis on composite scores. 
\end{remark}

\begin{definition}[divergence]
 Let $S$ be a composite score on $\mathcal{F}$. Then, we call 
 \begin{align*}
  D(f,g)=S(f,g)-S(f,f),\quad{} f,g\in\mathcal{F}, 
 \end{align*}
 the divergence associated with $S$. 
\end{definition}
By the definition of the composite score, 
the divergence $D(f,g)$ is nonnegative for all $f,g\in\mathcal{F}$, 
and the equality $D(p,q)=0$ for $p,q\in\mathcal{P}$ implies $p=q$. 

\subsection{Bregman scores}
\label{subsec:Bregman_scores}

As an important class of composite scores, we introduce a Bregman score and its separable variant. 
Under a mild assumption, any strictly proper score on $\mathcal{P}$ is expressed as a
Bregman score on $\mathcal{P}$ 
\cite{abernethy12:_charac_scorin_rules_linear_proper,%
gneiting07:_stric_proper_scorin_rules_predic_estim,hendrickson71:_proper_scores_probab_forec}. 
\begin{definition}[Bregman score]
 For a convex set $\mathcal{F}\subset{}L_0^+$, 
 let us define $G:\mathcal{F}\rightarrow\Rbb$ as a convex function such that 
 $G$ is strictly convex on $\mathcal{P}=\{p\in\mathcal{F}\,|\,\<p\>=1\}$. 
 Suppose that there exists a function $G^*_g:\Omega\rightarrow\Rbb$
 depending on 
 $g\in\mathcal{F}$ such that
 \begin{align*}
  G(f)\geq{}
  G(g)
  +\int_{\Omega}G^*_g(\omega)f(\omega)dm(\omega)
  -\int_{\Omega}G^*_g(\omega)g(\omega)dm(\omega),
  \quad\text{for}\ f,g\in\mathcal{F} 
 \end{align*}
 holds, where the integrals are assumed to be finite. 
 Then, 
 the Bregman score $S(f,g)$ on $\mathcal{F}$ is defined as 
 \begin{align*}
  S(f,g)=-G(g)
  -\int_{\Omega}G^*_g(\omega)f(\omega)dm(\omega)
  +\int_{\Omega}G^*_g(\omega)g(\omega)dm(\omega),
  \quad\text{for}\ f,g\in\mathcal{F}. 
\end{align*}
 The function $G$ is referred to as the potential function of the Bregman score, and 
 it satisfies $G(f)=-S(f,f)$. 
 The Bregman divergence is the divergence associated with the Bregman score. 
\end{definition}

From the definition, the Bregman score satisfies $S(f,g)\geq{}S(f,f)$ for all
$f,g\in\mathcal{F}$. 
The strict convexity of $G$ on $\mathcal{P}$ ensures that the Bregman score satisfies the
third condition of Definition~\ref{def:scoring-rule}; see Theorem~1 of 
\cite{gneiting07:_stric_proper_scorin_rules_predic_estim}. 
The function $G_g^*$ corresponds to the subgradient of $G$ at $g\in\mathcal{F}$. 
The rigorous definition of $G_g^*$ requires the dual space of a Banach space in $L_0^+$.
See \cite[Chap.~4]{borwein05:_techn} for sufficient conditions of the existence of $G^*_g$. 
To avoid technical difficulties, we assume the existence of $G^*_g$ in the above 
definition. 
The Bregman score is represented as the composite score 
\eqref{eqn:scoring-rule-expression} 
with 
\begin{align*}
 S_0(\omega,g)=-G_g^*(\omega)\ \ \text{and}\ \ 
 T(c,g)=c-G(g)+\<G_g^*\, g\>. 
\end{align*}
When the Bregman score is defined on the set of probability densities, setting
$S_0(\omega,g)=-G_g^*(\omega)-G(g)+\<G_g^*\, g\>$ and $T(c,g)=c$ is also a valid
choice. This implies that the Bregman score on $\mathcal{P}$ is represented as a strictly
proper score. 

The separable variant of the Bregman score is defined as follows. 
\begin{definition}
 [separable Bregman score]
 Let $J:\Rbb_+\rightarrow\Rbb$ be a strictly convex function. 
 The Bregman score with the potential function $G(f)=\<J(f)\>$ is called 
 the separable Bregman score. 
 The separable Bregman divergence is the divergence associated with the separable Bregman
 score. 
\end{definition}
 The separable Bregman score is of the form 
 \begin{align*}
  S(f,g)
  =-\<J(g)\>-\<J'(g)f\>+\<J'(g)g\>
  \qquad  \text{for}\ f,g\in\mathcal{F}, 
 \end{align*}
 where $J'(z)$ is the subgradient of $J$ at $z\in\Rbb_+$.

We show some examples of Bregman scores and associated divergences. 
\begin{example}
 [Kullback-Leibler (KL) score]
 \label{example:KL-div}
 Let $\mathcal{F}$ be a subset of $L_1^+$, and suppose that $f\log{g}$ is integrable for
 all $f,g\in\mathcal{F}$. 
 The Kullback-Leibler(KL) score is defined as 
 \begin{align*}
  S(f,g)=\<-f\log{g}+g\>, \qquad f,g\in\mathcal{F}, 
 \end{align*}
 which is the separable Bregman score using the function $J(z)=z\log{z}-z$ and 
the potential function $G(f)=\<f\log{f}-f\>$. 
 The associated divergence is called the KL divergence. 
\end{example}

\begin{example}
 [Density power score]
 \label{example:density-power-div}
 Let $\mathcal{F}$ be $\mathcal{F}=L_{1+\gamma}^+$ for a given $\gamma>0$. 
 The density power score 
 on $\mathcal{F}$ is defined as 
 \begin{align*}
  S(f,g)=
  \<g^{1+\gamma}\>-\frac{1+\gamma}{\gamma}\<fg^\gamma\>, 
  \qquad{}f,g\in\mathcal{F}, 
 \end{align*}
 which is the separable Bregman score with $J(z)=z^{1+\gamma}/\gamma$ and the potential
 function $G(f)=\<f^{1+\gamma}\>/\gamma$. 
 The integrability of $fg^{\gamma}$ is confirmed by H\"{o}lder's inequality. 
 The associated divergence is called the density power divergence 
 \cite{a.98:_robus_effic_estim_minim_densit_power_diver,basu10:_statis_infer,jones01:_compar}. 
 When the parameter $\gamma$ in the density power divergence tends to zero, 
 the KL-divergence is recovered. 
\end{example}

\begin{example}
 [$\gamma$-score; pseudospherical score]
 \label{example:gamma-div}
 Let $\mathcal{F}$ be $\mathcal{F}=L_{1+\gamma}^+$ for a given $\gamma>0$.  
 The pseudospherical score \cite{good71:_commen_measur_infor_uncer_r}
 is defined as 
 \begin{align*}
  S(f,g)=
  -\frac{\<fg^\gamma\>}{\<g^{1+\gamma}\>^{\gamma/(1+\gamma)}}, 
  \qquad{}f,g\in\mathcal{F}, 
 \end{align*}
 which is the non-separable Bregman score with the potential function
 $G(f)=\<f^{1+\gamma}\>^{1/(1+\gamma)}=\|f\|_{1+\gamma}$. 
 For the pseudospherical score $S(f,g)$, 
 the composite score $-\frac{1}{\gamma}\log(-S(f,g))$ is called the 
 $\gamma$-score in this paper. 
 The $\gamma$-score is proposed in 
 \cite{eguchi11:_projec_power_entrop_maxim_tsall_entrop_distr,fujisawa08:_robus}, 
 and it is used for robust parameter estimation. 
 As the limiting case of $\gamma\rightarrow0$, the divergence associated with the
 $\gamma$-score recovers KL-divergence. 
\end{example}

\subsection{Optimum score estimator}
\label{subsec:Statistical_inference_using_scores}

Statistical inference using the composite score 
\eqref{eqn:scoring-rule-expression} is conducted by substituting 
the empirical probability and the model probability into the composite score. 
Provided the i.i.d. samples $\omega_1,\ldots,\omega_n$
from the probability density $p$, an empirical approximation of $S(p,q)$ for a given
probability density $q$ is given as 
\begin{align*}
 S(\widetilde{p},q)=T\left(\frac{1}{n}\sum_{i=1}^{n}S_0(\omega_i,q),\,q \right), 
\end{align*}
where $\widetilde{p}$ denotes the empirical probability. 
For a sufficiently large number of samples, $S(\widetilde{p},q)$ converges to 
$S(p,q)$ due to the law of large numbers. 
Since $S(p,q)\geq{}S(p,p)$ is assumed, 
the estimator of $p$ is obtained as the minimum solution 
of $S(\widetilde{p},q)$ with respect to $q$ over a statistical model.
The estimator $\widehat{q}$ is called the \emph{optimum score estimator}
\cite{gneiting07:_stric_proper_scorin_rules_predic_estim}. 
The estimator using the strictly proper score is a special case of M-estimation
\cite{huber64:_robus}, and its statistical properties have been deeply investigated 
\cite{vaart00:_asymp_statis}. 

Different composite scores may produce the same estimator. 
Let us define the equivalence class on the set of composite scores 
such that the composite scores in the same
class provide the same estimator. 
\begin{definition}[equivalence of composite scores]
 \label{def:equivalence_of_scores}
 The composite scores $S(f,g)$ and $\widetilde{S}(f,g)$ on $\mathcal{F}$ are equivalent 
 if there exists a strictly increasing function $\xi:\Rbb\rightarrow\Rbb$ such
 that $\widetilde{S}(f,g)=\xi(S(f,g))$ holds for all $f,g\in\mathcal{F}$. 
 The composite scores $S(f,g)$ and $\widetilde{S}(f,g)$ on $\mathcal{F}$ are equivalent in
 probability
 if there exists a strictly increasing function $\xi:\Rbb\rightarrow\Rbb$ 
 such that $\widetilde{S}(p,q)=\xi(S(p,q))$ holds for all probability densities
 $p,q\in\mathcal{P}\subset\mathcal{F}$.
\end{definition}
For any strictly increasing function $\xi$, 
the minimum solutions of $S(p,q)$ and $\xi(S(p,q))$ with respect to $q$ are the same.  
Hence, the composite scores that are equivalent in probability provide the same estimator. 
A different definition of the equivalence class was also proposed by 
\cite{dawid98:_coher_measur_discr_uncer_depen,dawid07}, 
in which the Bregman scores $S(p,q)$ and $\widetilde{S}(p,q)$ on $\mathcal{P}$
are equivalent if there exist a positive constant $c>0$ and a function 
$k:\mathcal{P}\rightarrow\Rbb$ such that $\widetilde{S}(p,q)=c{S}(p,q)+k(p)$ holds. 
The equivalence class in Definition~\ref{def:equivalence_of_scores} is more suitable for 
our analysis.

\section{H\"{o}lder scores}
\label{sec:Holder-score}

In this section, we propose a class of composite scores, named H\"{o}lder scores, 
a part of which is not represented as the Bregman score. 
We investigate the relation between the H\"{o}lder scores and Bregman scores.

\subsection{Definition of H\"{o}lder score}
\label{subsec:Definition_of_Holder_score}

Bregman scores are widely used for statistical inference, 
\cite{banerjee05:_clust_bregm_diver,%
Collins_etal00,murata04:_infor_geomet_u_boost_bregm_diver,%
tsuda05:_matrix_expon_gradien_updat_learn_bregm_projec},
since 
one can substitute the empirical probability distribution into the Bregman score. 
Under a regularity condition, Bregman scores produce statistically consistent estimators
based on the outcomes. 
Especially, the density power score and $\gamma$-score are used for robust estimation
\cite{a.98:_robus_effic_estim_minim_densit_power_diver,fujisawa08:_robus}. 
In this section, we propose a class of composite scores called \emph{H\"{o}lder scores} 
that include both the density power score and $\gamma$-score. 
One can also substitute the empirical probability distribution into the H\"{o}lder score. 
As shown later, 
the H\"{o}lder score is not included in the class of Bregman scores, and has 
a relation to affine invariant estimators. 

\begin{definition}[H\"{o}lder score]
\label{def:Holder-score}
 The H\"{o}lder score with a nonnegative parameter $\gamma$ is defined as follows:
 \begin{enumerate}
  \item For a given $\gamma>0$, let $\phi:\Rbb_+\rightarrow\Rbb$  be a function such that
	$\phi(z)\geq-z^{1+\gamma}$ for all $z\geq0$ and $\phi(1)=-1$ hold. 
	Then, for $\mathcal{F}=L_{1+\gamma}^+$, 
	the H\"{o}lder score is defined as 
	\begin{align*}
	 S(f,g)=
	 \phi\bigg(\frac{\<fg^\gamma\>}{\<g^{1+\gamma}\>}\bigg)\<g^{1+\gamma}\>,\quad
	 f,g\in\mathcal{F}. 
	\end{align*}
  \item For $\gamma=0$, the H\"{o}lder score is defined as 
	\begin{align*}
	 S(f,g)=\<-f\log{g}+g\>, \quad f,g\in\mathcal{F}, 
	\end{align*}
	where $\mathcal{F}$ is a subset of $L_1^+$ such that
	$f\log{g}$ is integrable for all $f,g\in\mathcal{F}$. 
 \end{enumerate}
 The associated divergence $D(f,g)=S(f,g)-S(f,f)$ is called the H\"{o}lder divergence. 
 \end{definition}

The H\"{o}lder score with $\gamma=0$ is nothing but the KL score. 
An appropriate choice of the function $\phi$ produces the composite score equivalent with
the density power score or $\gamma$-score. 
Indeed, the H\"{o}lder score with the lower bound $\phi(z)=-z^{1+\gamma}$ is 
$S(f,g)=-\<fg^\gamma\>^{1+\gamma}/\<g^{1+\gamma}\>^\gamma$
which is equivalent with $\gamma$-score. 
The density power score is equivalent with the H\"{o}lder score with
$\phi(z)=\gamma-(1+\gamma)z$. 

We prove the basic property that the H\"{o}lder score satisfies the condition of 
the composite score in Definition~\ref{def:scoring-rule}. 
\begin{theorem}
 \label{theorem:Holder-score-well-def}
 The H\"{o}lder score is a composite score. 
\end{theorem}

The proof of Theorem~\ref{theorem:Holder-score-well-def} is found in 
Appendix~\ref{appendix:theorem:Holder-score-well-def}. 
The H\"{o}lder score with $\gamma>0$ is represented as the composite score 
\eqref{eqn:scoring-rule-expression} with $S_0(\omega,g)=g(\omega)^{\gamma}$ and 
$T(c,g)=\phi(c/\<g^{1+\gamma}\>)\<g^{1+\gamma}\>$. 
The name of H\"{o}lder score comes from the fact that H\"{o}lder's inequality is used to
prove the non-negativity of H\"{o}lder divergence. 
The function $S(f,f)$ is referred to as \emph{entropy}. 
The entropy of the H\"{o}lder score is $S(f,f)=-\<f^{1+\gamma}\>$, which is 
in agreement with the Tsallis entropy \cite{tsallis88:_possib_boltz_gibbs}
up to an affine transformation.

\subsection{Bregman scores and H\"{o}lder scores}
\label{subsec:Bregman_score_and_Holder_score}

Let us consider the relation between the Bregman scores and H\"{o}lder scores. 
We assume the differentiability for Bregman scores. 
The definition of the differentiability is shown below. 
\begin{definition}
 [differentiability of potential function] 
 Let $G$ be the potential function of the Bregman score 
 on the convex set $\mathcal{F}$. 
 If the limit 
 \begin{align*}
 \lim_{\varepsilon\rightarrow{0}}
 \frac{G((1-\varepsilon)f+\varepsilon{g})-G(f)}{\varepsilon} 
 \end{align*} 
 exists for any $f,g\in\mathcal{F}$ such that 
 there exists $\delta>0$ satisfying
 $(1-\varepsilon)f+\varepsilon{g}\in\mathcal{F}$ for all
 $\varepsilon\in(-\delta,\delta)$, 
 the potential function $G$ is differentiable. 
 The corresponding Bregman score (resp. divergence) is called the differentiable Bregman score
 (resp. divergence). 
\end{definition}
The differentiability above makes our analysis rather simple. 
For non-differentiable Bregman scores, we will need more involved argument such as the
convex analysis in Banach spaces. 
From the practical viewpoint, differentiable Bregman scores 
will be preferable, since the standard non-linear optimization techniques are directly
applicable to obtain the optimum score estimator. 

\begin{theorem}
 \label{theorem:Bregman-Holder}
 Suppose that the function $\phi$ in the H\"{o}lder score is continuous on $\Rbb_+$. 
 \begin{enumerate}
  \item Suppose that 
	the differentiable Bregman score with the potential function $G(f)$
	is equivalent with the H\"{o}lder score with $\gamma>0$. 
	Then, $G(f)$ is given as $G(f)=\<f^{1+\gamma}\>^{\kappa/(1+\gamma)}$ 
	up to a positive constant factor, where $\kappa\geq1$. 
  \item Suppose that the differentiable and separable Bregman score with the potential 
	function $G(f)$ is equivalent with the H\"{o}lder score with $\gamma>0$. 
	Then, $G(f)$ is given as $G(f)=\<f^{1+\gamma}\>$
	up to a positive constant factor. 
 \end{enumerate}
\end{theorem}
The proof is shown in Appendix \ref{appendix:proof_Bregman_Holder}. 

The KL score is a differentiable and separable Bregman. 
Hence, the intersection of (separable) Bregman score and H\"{o}lder score is the KL score
or the (separable) Bregman score associated with the potential function presented in the
above theorem. 

For the potential function $G(f)=\<f^{1+\gamma}\>^{\kappa/(1+\gamma)}$ with $\gamma>0$ and
$\kappa\geq1$, 
the corresponding Bregman score is given as 
\begin{align}
 \label{eqn:Bregman-Holder-score}
 S(f,g)=\<g^{1+\gamma}\>^{\kappa/(1+\gamma)}\left(
 1-\frac{1}{\kappa}-\frac{\<fg^\gamma\>}{\<g^{1+\gamma}\>}
 \right). 
\end{align}
The above Bregman scores include the density power score ($\kappa=1+\gamma$)
and $\gamma$-score ($\kappa=1$) in each equivalent class. 
The H\"{o}lder score corresponding to the Bregman score \eqref{eqn:Bregman-Holder-score}
is given by the function $\phi(z)$ defined as 
\begin{align}
 \label{eqn:Bregman-Holder-score-phi}
 \phi(z)=-\kappa^{(1+\gamma)/\kappa}|z-1+1/\kappa|^{(1+\gamma)/\kappa}
 \mathrm{sign}(z-1+1/\kappa),
\end{align}
where $\mathrm{sign}(z)$ is the sign function 
taking $z/|z|$ for $z\neq0$ and $0$ for $z=0$. 
In Section \ref{subsec:Asymptotically_unbiased_estimation_for_regression_problems}, 
we show a statistical interpretation of the composite scores included
in the intersection of Bregman scores and H\"{o}lder scores.

\section{Affine invariance of H\"{o}lder divergence}
\label{sec:Affine_invariance_Holder-div}

Affine transformation of the observed data is often used in statistical analysis. 
Let $\Omega=\Rbb^d,\, \mathcal{B}$ be the Borel set of $\Omega$, and $m$ be the Lebesgue
measure on $(\Omega,\mathcal{B})$. 
The affine transformation is defined as the map $\omega\mapsto\sigma^{-1}(\omega-\mu)$ of
$\omega\in\Omega$ with an invertible matrix $\sigma\in\Rbb^{d\times{d}}$ and 
a vector $\mu\in\Rbb^d$. 
The normalization is a typical example of the affine transformation. 
For the observed data $\omega_1,\ldots,\omega_n\in\Rbb^d$, 
let the vector $\mu$ be the sample mean of the observations, and 
the matrix $\sigma$ be the diagonal matrix such that the $k$-th diagonal element is 
equal to the sample-based standard deviation of the $k$-th component of the observed data. 
Then, each element of the transformed data,
$\sigma^{-1}(\omega_1-\mu),\ldots,\sigma^{-1}(\omega_n-\mu)$, has zero sample-mean and
unit sample-variance. 
This transformation enables the fair comparison of the intensity of each component 
in statistical sense. 
As another benefit, the normalization often makes the numerical computation stable. 

The affine transformation of data, $\omega\mapsto\sigma^{-1}(\omega-\mu)$, 
induces the transformation of the probability density, 
\begin{align*}
 p(\omega)\mapsto{}p_{\sigma,\mu}(\omega)=|\!\det{\sigma}|p(\sigma{\omega}+\mu). 
\end{align*}
Let $q$ be a statistical model to estimate the probability density $p$. 
Then, the statistical model for the affine transformed data is given as $q_{\sigma,\mu}$. 
Let $\widehat{q}$ be the estimator of $p$ based on the original data
$\{\omega_1,\ldots,\omega_n\}$,
and $\widehat{q_{\sigma,\mu}}$ be the estimator based on the transformed data, 
$\{\sigma^{-1}(\omega_1-\mu),\ldots,\sigma^{-1}(\omega_n-\mu)\}$. 
It will be natural to require that the estimator is transformed in a consistent way, 
when the data is transformed, i.e., the equality 
\begin{align}
 \label{eqn:equivariant-estimator}
 (\,\widehat{q}\,)_{\sigma,\mu}=\widehat{q_{\sigma,\mu}}
\end{align}
should hold. 
The estimators enjoying \eqref{eqn:equivariant-estimator} do not essentially depend on the
choice of the units in the measurement. 
In the present paper, the estimator satisfying \eqref{eqn:equivariant-estimator} is called 
the \emph{affine invariant estimator}.
In a formal mathematical description, the term \emph{equivariant estimator} is used to denote
the estimator that changes in a consistent way under data transformations
\cite{berger85:_statis_decis_theor_bayes_analy}. 

A simple way of obtaining the affine invariant estimator is to use 
the composite scores satisfying the equality $S(p,q)=S(p_{\sigma,\mu},q_{\sigma,\mu})$. 
However, the equality is not necessity. 
In the below, we introduce composite scores and associated divergences 
that provide the affine invariant estimator. 

\begin{definition}[affine invariant divergence; affine invariant composite score]
 \label{def:affine_invariance}
 Let $S$ be a composite score on $\mathcal{F}$, and $D$ be the associated divergence. 
 The divergence $D(f,g)$ is affine invariant if there exists an $\Rbb_+^\circ$-valued
 function $h(\sigma,\mu)$ of the invertible matrix $\sigma\in\Rbb^{d\times{d}}$ and the
 vector $\mu\in\Rbb^d$ such that the equality 
 \begin{align}
  h(\sigma,\mu)D(p_{\sigma,\mu},q_{\sigma,\mu})=D(p,q)
  \label{eqn:quasi-invariance-divergence}
 \end{align}
 holds for any pair of probability densities $p,q\in\mathcal{P}$ and arbitrary affine 
 transformation with $(\sigma,\mu)$. The function $h$ is called the scale function. 
 The composite score $S$ inducing the affine invariant divergence is 
 called the affine invariant composite score. 
\end{definition}

We briefly prove that the affine invariant composite score provides 
the affine invariant estimator. 
Let $S$ be an affine invariant composite score, and $\widehat{q}$ be the optimum score
estimator obtained by solving the minimization problem $\min_{q\in\mathcal{M}}S(p,q)$ 
on a statistical model $\mathcal{M}$. 
Then, the inequalities, 
\begin{align*}
 D(p,\widehat{q})\leq{}D(p,q)\quad\text{and}\quad
 D(p_{\sigma,\mu},(\,\widehat{q}\,)_{\sigma,\mu})\leq{}D(p_{\sigma,\mu},q_{\sigma,\mu})
\end{align*}
hold for all $q\in\mathcal{M}$. 
On the other hand, $\widehat{q_{\sigma,\mu}}$ is the
minimum solution of $\min_{q_{\sigma,\mu}}D(p_{\sigma,\mu},q_{\sigma,\mu})$, 
when the model $\{q_{\sigma,\mu}\,|\,q\in\mathcal{M}\}$ is used. 
Therefore, the equivariant property \eqref{eqn:equivariant-estimator} holds, if the
optimal solution is unique.

It is straightforward to verify that the H\"{o}lder score is affine invariant. 
Indeed, for the H\"{o}lder divergence $D(p,q)$ with $\gamma>0$, we have 
\begin{align*}
 D(p_{\sigma,\mu},q_{\sigma,\mu})
 &=
 \phi\left(
 \frac{\<p_{\sigma,\mu}q_{\sigma,\mu}^\gamma\>}{\<q_{\sigma,\mu}^{1+\gamma}\>}
 \right)\<q_{\sigma,\mu}^{1+\gamma}\>+\<p_{\sigma,\mu}^{1+\gamma}\>\\
 &=
 \phi\left( \frac{|\det\sigma|^{\gamma}\<pq^\gamma\>}{|\det\sigma|^{\gamma}\<q^{1+\gamma}\>} \right)
 \<q^{1+\gamma}\>|\det\sigma|^{\gamma}+\<p^{1+\gamma}\> |\det\sigma|^{\gamma}\\
 &=|\det\sigma|^{\gamma} D(p,q). 
\end{align*}
Therefore, the scale function is given as $h(\sigma,\mu)=|\!\det\sigma|^{-\gamma}$. 
In the same way, we can confirm that the KL divergence is also affine invariant with the 
scale function $h(\sigma,\mu)=1$. 
This result indicates that the optimum score estimator using H\"{o}lder score provides the
affine invariant estimator. 

Conversely, we prove that the H\"{o}lder score is characterized by the affine invariance. 
In the beginning, let us introduce some assumptions. 
\begin{assumption}
 [basic assumption on $\Omega$ and $\mathcal{F}$]
 \label{assumption:measure-space_F}
 Let $\Omega=\Rbb^d,\, \mathcal{B}$ be the Borel set of $\Omega$, 
 and $m:\mathcal{B}\rightarrow\Rbb_+$ be the Lebesgue measure on $(\Omega,\mathcal{B})$.  
 The set $\mathcal{F}$ includes the following function set, 
 \begin{align*}
  \mathcal{F}_0:=\bigg\{f\in{L_0^+}\,\bigg|\,
  \begin{array}{l}
   \displaystyle
    \{\omega\in\Omega\,|\,f(\omega)>0\}=(0,1)^d, \,
    \text{and there exist $a, b\in\Rbb$}   \\ 
   \displaystyle
    \text{such that $0<a<f(\omega)<b$ for all $\omega\in(0,1)^d$.}
  \end{array}
  \bigg\}, 
 \end{align*}
i.e., $\mathcal{F}_0\subset\mathcal{F}\subset{}L_0^+$ holds. 
\end{assumption}
The subset $(0,1)^d$ in the above assumption can be replaced with any subset with a finite measure. 
We assume the following conditions on the composite score. 
\begin{assumption} 
 [assumption on the composite score] 
 \label{assumption:psi-U-V}
 For the composite score, we assume three conditions: 
 \begin{description}
  \item[(a)] The composite score $S(f,g)$ has the form of 
	     \begin{align}
	      S(f,g)=\psi(\<fU(g)\>,\<V(g)\>)\quad \text{for all}\ f,g\in\mathcal{F},
	      \label{eqn:special-score}
	     \end{align}
	     where $U$and $V$ are real-valued functions on $\Rbb_+$ and 
	     $\psi$ is a function on a subset of $\Rbb^2$, 
	     i.e., $S$ is the composite score 
	     \eqref{eqn:scoring-rule-expression} with $S_0(\omega,g)=U(g(\omega))$ and
	     $T(c,g)=\psi(c,\<V(g)\>)$. For all $f,g\in\mathcal{F}$, the functions $fU(g)$
	     and $V(g)$ are integrable. 
  \item[(b)] The functions $U,\,V:\Rbb_{+}\rightarrow\Rbb$ 
	     are second order continuously differentiable on $\Rbb_{+}^\circ$, and 
	     they are not constant function on $\Rbb_{+}^\circ$. 
	     For the function $V$, the equality $\displaystyle\lim_{z\searrow0}V(z)=0=V(0)$
	     holds, and the limit $\displaystyle\lim_{z\searrow0}V'(z)$ exists. 
  \item[(c)] Let $D_{U,V}$ and $E_{U,V}$ 
	     be subsets of $\Rbb^2$ defined as
	     \begin{align*}
	      D_{U,V}&=\{(\<fU(g)\>,\<V(g)\>)\in\Rbb^2\,|\,f,g\in\mathcal{F}\},\\
	      E_{U,V}&=\{(\<fU(f)\>,\<V(f)\>)\in\Rbb^2\,|\,f\in\mathcal{F}\},
	     \end{align*}
	     respectively. 
	     For arbitrary point $x\in{D_{U,V}}$, there exists an open 
	     neighbourhood of $x$ on which $\psi$ is second order continuously
	     differentiable. 
	     For arbitrary point $x\in{E_{U,V}}$, there exists an open 
	     neighbourhood of $x$ on which the gradient vector $(\psi_1,\psi_2)$ does not
	     vanish. 
 \end{description}
\end{assumption}

All separable Bregman scores are expressed as the form of \eqref{eqn:special-score}. 
There exist Bregman scores that are not described by \eqref{eqn:special-score}, 
while Bregman scores do not cover all the composite scores \eqref{eqn:special-score}. 
The composite score of the form \eqref{eqn:special-score} is useful in practice, 
since it can be calculated via integrals. 
In Assumption~\ref{assumption:psi-U-V} (b), we assumed $V(0)=0$ in order to guarantee the
integrability of the function whose support is not equal to $\Omega$. 
More precisely, let $Z=\{\omega\in\Rbb^d\,|\,f(\omega)=0\}$ with
$m(Z)=\infty$, then 
$\<V(g)\>=\int_{Z}V(0)dm+\int_{\Omega\setminus{Z}}V(g)dm$ 
will not be finite unless $V(0)=0$. 
In Assumption~\ref{assumption:psi-U-V}\,(c), we assumed that the gradient vector
$(\psi_1,\psi_2)$ does not become the zero vector at $(\<fU(f)\>,\<V(f)\>)$. 
If this assumption does not hold, we need a more involved argument to derive analytic 
properties of the functions $U$ and $V$. 
For the sake of simplicity, we introduce Assumption~\ref{assumption:psi-U-V}\,(c). 

The functions $U$ and $V$ of the affine invariant composite scores are determined by 
Theorem~\ref{theorem:affine-invariance_UV}. 
\begin{theorem}
 \label{theorem:affine-invariance_UV}
 Let $S$ be a composite score that produces an affine invariant divergence. 
 Suppose that Assumption~\ref{assumption:measure-space_F} and 
 Assumption~\ref{assumption:psi-U-V} hold. 
 Then, the functions $U$ and $V$ in \eqref{eqn:special-score} are given as
 $U(z)=z^\gamma+c$ and $V(z)=z^{1+\gamma}$ with $\gamma>0$, or
 $U(z)=-\log{z}+c$ and $V(z)=z$ up to a constant factor, where $c\in\Rbb$ is a constant. 
\end{theorem}

The proof is found in Appendix~\ref{appendix:proof_UV}. 
For each possibility of $U$ and $V$, 
the composite score is identified in the following theorem. 
\begin{theorem}
 \label{theorem:psi_form}
 Let $S$ be a composite score that produces an affine invariant divergence. 
 Suppose that Assumption \ref{assumption:measure-space_F} and 
 Assumption \ref{assumption:psi-U-V} hold. 
\begin{enumerate}
 \item \label{case-1}
       Let us define $U(z)=-\log{z}+c$ and $V(z)=z$ in \eqref{eqn:special-score}. 
       Then, the composite score $S(f,g)$ is equivalent in probability with 
       the {\rm{KL}} score. 
 \item \label{case-2}
       For $\gamma>0$, 
       let us define $U(z)=z^\gamma+c$ and $V(z)=z^{1+\gamma}$ in \eqref{eqn:special-score}, 
       and let $\mathcal{F}$ be $\mathcal{F}=L_{1+\gamma}^+$. 
       Then, the composite score $S(f,g)$ is equivalent in probability with 
       the H\"{o}lder score with $\gamma>0$ and a function $\phi$. 
\end{enumerate}
\end{theorem}
The proof is found in Appendix~\ref{appendix:subsec:proof-affine-invariant-psi}. 

In the first case of Theorem \ref{theorem:psi_form}, 
the integrability of $f\log{g}$ is assumed for $f,g\in\mathcal{F}$ 
such that $\mathcal{F}_0\subset\mathcal{F}$, 
implying that $\{\omega\in\Omega\,|\,f(\omega)>0\}=(0,1)^d$ holds for $f\in\mathcal{F}$. 

Theorem~\ref{theorem:Bregman-Holder} and Theorem~\ref{theorem:psi_form} 
imply that the density power score is characterized by the differentiable, 
separable and affine invariant Bregman score. 
Indeed, the score of the form \eqref{eqn:special-score} includes the differentiable and separable Bregman score, 
and the affine invariant score of the form \eqref{eqn:special-score} is H\"{o}lder score. 
As shown in Theorem~\ref{theorem:Bregman-Holder},
the intersection of the differentiable and separable Bregman scores and the H\"{o}lder 
scores is the density power score.

%

\section{Applications of H\"{o}lder scores}
\label{sec:Statistical_Inference_using_Holder_score}

We use H\"{o}lder scores for regression and robust estimation, 
and investigate the corresponding statistical properties.

\subsection{Asymptotically unbiased estimation for regression problems}
\label{subsec:Asymptotically_unbiased_estimation_for_regression_problems}
We use a composite score for the estimation of conditional probabilities or regression functions. 
Let $x$ and $y$ be the explanatory variable and objective variable, respectively. 
Suppose that the i.i.d. samples $(x_i,y_i),\,i=1,\ldots,n$ are observed from the joint
probability density $p(y|x)r(x)$, where $p(y|x)$ is the conditional probability density
of $y$ given $x$ and $r(x)$ is the marginal 
probability density of $x$. Our concern is to estimate $p(y|x)$ from the samples, and the 
estimation of the marginal probability $r(x)$ is not required. 

To estimate $p(y|x)$, let us define a statistical model $\mathcal{M}$, which is a set of
conditional probability densities. 
Suppose that $p(y|x)$ is realized by the model $\mathcal{M}$, i.e.,
$p(y|x)\in\mathcal{M}$. 
On each input vector $x$, the discrepancy between $p(y|x)$ and $q(y|x)\in\mathcal{M}$ is measured by  
$S(p(\cdot|x),q(\cdot|x))$, where $S$ is a composite score. 
By averaging the composite score with respect to
the marginal distribution, we obtain the averaged composite score
\begin{align}
 \bar{S}(p,q|r):=\int{}\!S(p(\cdot|x),q(\cdot|x))r(x)dm(x)
 \label{eqn:Holder-regression-estimator}. 
\end{align}
which is regarded as the loss of the estimate
$q(y|x)\in\mathcal{M}$ under the probability density $p(y|x)r(x)$. 
From the definition of the composite score, 
the minimum solution of the averaged composite score with respect to $q\in\mathcal{M}$ is
attained at $q(y|x)=p(y|x)$. 

Let us consider the empirical approximation of $\bar{S}(p,q|r)$ 
in \eqref{eqn:Holder-regression-estimator}. 
If $\bar{S}(p,q|r)$ is represented as the expectation with respect to the joint
probability $p(y|x)r(x)$, $\bar{S}(p,q|r)$ can be approximated by the 
empirical mean of the samples, $\{(x_1,y_1),\ldots,(x_n,y_n)\}$. 
Otherwise, we need an estimate of the conditional probability $p(y|x)$ to obtain an
approximation of $\bar{S}(p,q|r)$. 
Clearly, the later case is not practical, since our purpose is to estimate $p(y|x)$. 

Suppose that for any $r(x)$, the averaged composite score $\bar{S}(p,q|r)$ is represented
as the expectation for the probability $p(y|x)r(x)$. Then, $S$ is a Bregman score, i.e,
$S(p(\cdot|x),q(\cdot|x))$ is expressed as the expectation with respect to
$p(\cdot|x)$. If the Bregman score that is equivalent in probability with the H\"{o}lder
score is used, the affine invariant estimator is obtained for the estimation of the
conditional probability. 
Here, the affine transformation of the objective variable is considered. 

Theorem \ref{theorem:Bregman-Holder} shows that 
the Bregman score that is equivalent in probability with the H\"{o}lder score 
is of the form \eqref{eqn:Bregman-Holder-score}. 
The optimum score estimator using \eqref{eqn:Bregman-Holder-score} is the minimum solution of 
\begin{align}
 \min_{q\in\mathcal{M}}\ 
 \frac{1}{n}\sum_{i=1}^{n}\<q(\cdot|x_i)^{1+\gamma}\>^{\kappa/(1+\gamma)}
 \bigg( 1-\frac{1}{\kappa}-\frac{q(y_i|x_i)^\gamma}{\<q(\cdot|x_i)^{1+\gamma}\>} \bigg), 
 \label{eqn:Holder-score-for-regression}
\end{align}
where $\gamma>0$ and $\kappa\geq1$. 
The composite score \eqref{eqn:Holder-score-for-regression} provides the Fisher consistent
estimator of the conditional probability. 
The estimator with the density power score (resp. $\gamma$-score) is obtained by setting 
$\kappa=1+\gamma$ (resp. $\gamma=1$). 
Though a general family of scores including the density power score and $\gamma$-score was
proposed by \cite{cichocki10:_famil_alpha_beta_gamma_diver}, 
the score \eqref{eqn:Bregman-Holder-score} is different from the existing family.

The estimator \eqref{eqn:Holder-score-for-regression} 
is the equivariant estimator under the affine transformation. 
Provided the data $(x_i,y_i), i=1,\ldots,n$, 
let $(\xi(x_i),\,\sigma^{-1}(y_i-\mu)), i=1,\ldots,n$ 
be the transformed data, 
where $\xi$ is a one-to-one mapping and 
$\sigma^{-1}(y-\mu)$ is the affine transformation of $y$. 
When the model $|\!\det{\sigma}|\,q(\sigma{y}+\mu\,|\,\xi(x))$ defined from 
$q\in\mathcal{M}$ is used to the transformed data, 
the estimator is given by 
$|\!\det{\sigma}|\,\widehat{q}(\sigma{y}+\mu\,|\,\xi(x))$, where
$\widehat{q}(y|x)$ is the estimator obtained by \eqref{eqn:Holder-score-for-regression} 
based on the original data.

\subsection{Robust estimation using H\"{o}lder scores}
\label{subsec:Robust_estimation_using_Holder_score}

The Bregman scores such as the density power scores and $\gamma$-scores are used for 
robust estimation 
\cite{a.98:_robus_effic_estim_minim_densit_power_diver,fujisawa08:_robus}.  
Let us consider the robustness property of H\"{o}lder scores. 
In robust statistics, the main concern is to develop statistical methods that are not
affected by outliers or other small departures from model assumptions. 

The robustness of the estimator is quantified by the breakdown point, influence function
and so forth \cite{Hampel_etal86}. 
Here, the influence function is used to analyze the robustness of the optimum score estimators. 
Let us introduce the influence functions briefly. 
Let $p_\theta(x)$ be a probability density on $\Rbb^d$ with a finite dimensional parameter 
$\theta\in\Theta\subset\Rbb^k$, and $\delta_z(x)$ be the probability density having a
point mass at $x=z$. 
Given the probability density $p_\varepsilon(x)=(1-\varepsilon)p_\theta(x)+\varepsilon\delta_{z}(x)$, 
let $\theta_\varepsilon$ be the minimizer of
$\min_{\bar{\theta}\in\Theta}\, S(p_\varepsilon,p_{\bar{\theta}})$, 
where $S$ is a composite score. 
For $\varepsilon=0$, the optimal solution is $\theta_0=\theta$. 
The parameter $\theta_\varepsilon$ is the optimum score estimator under the contamination
$\delta_z$. 
The influence function of the optimum score estimator against the contamination $\delta_z$
is defined as  
\begin{align*}
 \mathrm{IF}(z;\theta,S)
 =\lim_{\varepsilon\rightarrow+0}\frac{\theta_\varepsilon-\theta}{\varepsilon}. 
\end{align*}
The influence function $\mathrm{IF}(z;\theta,S)$ provides
several measures of the robustness for the optimum score estimator. 
An example is the gross error sensitivity $\sup_{z}\|\mathrm{IF}(z;\theta,S)\|$,
where $\|\cdot\|$ is the Euclidean norm. 
The estimator that uniformly minimizes the gross error sensitivity 
over the parameter space is called the most B(ias)-robust estimator.  
The most B-robust estimator minimizes the worst-case influence of outliers. 
For the one-dimensional normal distribution, the median estimator is the most 
B-robust for the estimation of the mean value \cite{Hampel_etal86}. 
On the other hand, the estimator satisfying 
\begin{align*}
 \lim_{\|z\|\rightarrow\infty}\|\mathrm{IF}(z;\theta,S)\|=0
 \quad \text{for all}\ \theta\in\Theta
\end{align*}
is called the \emph{redescending estimator} \cite{Hampel_etal86,maronna06:_robus_statis}. 
The redescending property is preferable for stable inference, 
since the influence of extreme outliers tends to zero. 
Note that the most B-robust estimator is not necessarily  the redescending estimator, 
and vice versa. 

It is known that under the normal distribution, the $\gamma$-score has the redescending
property, while the density power score does not \cite{fujisawa08:_robus}. 
In the following theorem, we present the necessary and sufficient condition that 
the optimum score estimator using the H\"{o}lder score 
has the redescending property for general statistical models. 
\begin{theorem}
 \label{theorem:redescending_prop}
 Suppose that the function $\phi(z)$ in the H\"{o}lder score is second order continuously
 differentiable around $z=1$. 
 For the statistical model $p_\theta(x), \theta\in\Theta\subset\Rbb^k$, 
 let $s_\theta(x)\in\Rbb^k$ be the score function of the model, i.e., 
 $(s_\theta(x))_i=\frac{\partial}{\partial\theta_i}\log{}p_\theta(x),\, i=1,\ldots,k$. 
 Let us assume the following conditions: 
 \begin{enumerate}
  \item The limiting condition 
	$\displaystyle\lim_{\|z\|\rightarrow\infty}p_\theta(z)=0$ holds for all parameter $\theta$. 
  \item There exists $\gamma>0$ satisfying the followings: 
	\begin{enumerate}
	 \item 	$p_\theta\in{L_{1+\gamma}^{+}}$ holds for all $\theta$. 
	 \item  $\displaystyle\lim_{\|z\|\rightarrow\infty}p_\theta(z)^\gamma{}s_\theta(z)=0$
		holds for all parameter $\theta$. 
	 \item 	Let $I\in\Rbb^{k\times{k}}$ be the Hessian matrix of 
		$\phi(\<p_{\theta^*}p_\theta^\gamma\>/\<p_\theta^{1+\gamma}\>)\<p_\theta^{1+\gamma}\>$
		at $\theta=\theta^*\in\Theta$, i.e., 
		\begin{align}
		 I_{ij}=\frac{\partial^2}{\partial\theta_i\partial\theta_j}
		 \left\{
		 \phi\bigg(
		 \frac{\<p_{\theta^*}p_\theta^\gamma\>}{\<p_\theta^{1+\gamma}\>}
		 \bigg)\<p_\theta^{1+\gamma}\>
		 \right\}\bigg|_{\theta=\theta^*}, 
		 \label{eqn:Hessian_matrix}
		\end{align}
		for $i,j=1,\ldots,k$. 
		The Hessian matrix $I$ is invertible at any $\theta^*\in\Theta$. 
	 \item For any $\theta^*\in\Theta$, 
	       the integral under the measure $m$ 
	       and the differential with respect to $\theta$ 
	       for the functions $\<p_\theta^{1+\gamma}\>$ and $\<p_{\theta^*}p_\theta^{\gamma}\>$
	       are interchangeable in the vicinity of $\theta=\theta^*$. 
	       In addition, 
	       there exists a parameter $\theta$ such that the	integral
	       $\<p_\theta^{1+\gamma}s_\theta\>$ is not equal to the zero vector. 
	\end{enumerate}
 \end{enumerate}
 Then, the optimum score estimator using H\"{o}lder score with $\gamma>0$
 satisfies the redescending property 
 for arbitrary statistical model satisfying the above conditions
 if and only if $\phi''(1)=-\gamma(1+\gamma)$ holds. 
 All such estimators have the same asymptotic variance. 
\end{theorem}

The proof is deferred to Appendix~\ref{appendix:proof_redescending_property}. 

The H\"{o}lder score that is equivalent in probability with the $\gamma$-score 
satisfies $\phi''(1)=-\gamma(1+\gamma)$. 
Hence, for general parametric models, the optimum score estimator using $\gamma$-score has
the redescending property. 
The H\"{o}lder scores with $\phi''(1)=-\gamma(1+\gamma)$ include 
non-Bregman scores, implying that non-Bregman scores can be useful for statistical inference. 

The $\gamma$-score is characterized by the following three conditions, 
i) affine invariance, ii) applicability to regression problems, and iii) redescending
property. 
Indeed, the function $\phi$ in \eqref{eqn:Bregman-Holder-score-phi} satisfies
$\phi''(1)=-\gamma(1+\gamma)+(\kappa-1)(1+\gamma)$, and
$\phi''(1)=-\gamma(1+\gamma)$ holds only for $\kappa=1$, i.e., the case of
$\gamma$-score. 
A characterization of $\gamma$-score is also presented in \cite{fujisawa08:_robus}. 
Comparing to the argument in \cite{fujisawa08:_robus}, our characterization is more 
directly connected with the statistical properties of the optimum score estimator.

\section{Conclusion}
\label{sec:conclusion}

We introduced the H\"{o}lder score that is a class of composite scores, 
and presented its characterization based on the affine invariance of the associated
divergence. 
We studied the relation between the H\"{o}lder score and the conventional proper score, i.e.,
the Bregman score, and derived a class of Bregman scores that is represented as the
mixture form of the density power score and $\gamma$-score. 
We also found that the density power score is the intersection of the separable Bregman
scores and H\"{o}lder scores. 
Then, we used H\"{o}lder scores for statistical inference including regression problems and 
robust parameter estimation. 
The H\"{o}lder scores that are applicable to regression problems 
are given by the intersection of Bregman scores and H\"{o}lder scores. 
The H\"{o}lder scores outside of the intersection will not produce asymptotically unbiased
estimators for the regression problems. 
In robust parameter estimation, 
the redescending property was investigated for H\"{o}lder score. 
We proved that the H\"{o}lder score satisfying the mild condition on the function $\phi$
yields the robust estimator against extreme outliers. 
In the class of H\"{o}lder scores, only the $\gamma$-score provides 
the robust and asymptotically unbiased estimator for regression problems. 

As shown in robust estimation in 
Section~\ref{subsec:Robust_estimation_using_Holder_score}, 
the H\"{o}lder score other than Bregman score can be useful for statistical inference. 
In this paper, we focused on composite scores of the form~\eqref{eqn:special-score}. 
An expansion of \eqref{eqn:special-score} may provide a wider class of affine invariant
composite scores. 
The final goal on this line is to specify all the affine invariant composite scores, and
to reveal its statistical properties. 
It is also an interesting future work to identify the composite scores inducing 
equivariant estimators under a data-transformation other than the affine transformation. 
Another interesting research direction is 
to investigate the class of equivariant estimators defined from the proper local scores, 
which depend on the predictive density through its value and the values of its derivatives 
\cite{dawid12:_proper_local_scorin_rules_discr_sampl_spaces,ehm12:_local,parry12:_proper_local_scorin_rules}. 
The proper local scores provide practical estimators under large dimensional statistical
models, since they can be computed without knowledge of the normalizing constant of the
probability densities. 
The invariance of the proper local scores under data-transformations
is an important feature to understand the statistical properties of 
the associated estimators.

\appendix

\section{H\"{o}lder divergence}
\label{appendix:theorem:Holder-score-well-def}

\begin{proof}[proof of Theorem~\ref{theorem:Holder-score-well-def}] 

The H\"{o}lder score with $\gamma=0$ is the KL score, which is a strictly proper score as
shown by many authors. 
Let us consider H\"{o}lder score $S(f,g)$ with $\gamma>0$ defined on
$\mathcal{F}=L_{1+\gamma}^+$. 
Provided $f\in{}\mathcal{F}$ and $g^\gamma\in{}L_{1+1/\gamma}^+$ for $g\in{}\mathcal{F}$, 
the H\"{o}lder's inequality leads to 
\begin{align*}
 \<fg^\gamma\>\leq
 \<f^{1+\gamma}\>^{1/(1+\gamma)}\<g^{1+\gamma}\>^{\gamma/(1+\gamma)}\quad
 \text{for all}\  f,g\in\mathcal{F}. 
\end{align*}
The equality holds if and only if $f$ and $g$ are linearly dependent. 
From the inequality $\phi(z)\geq-z^{1+\gamma}$ for $z\geq0$, we have 
\begin{align*}
 S(f,g)-S(f,f)
 &=
 \phi\left(\frac{\<fg^\gamma\>}{\<g^{1+\gamma}\>}\right)\<g^{1+\gamma}\>+\<f^{1+\gamma}\>\\
 &\geq
 -\left(\frac{\<fg^\gamma\>}{\<g^{1+\gamma}\>}\right)^{1+\gamma} \<g^{1+\gamma}\>+\<f^{1+\gamma}\>\\
 &\geq0. \qquad\qquad\qquad\qquad \text{(H\"{o}lder's inequality)}
\end{align*}
Suppose that $S(p,q)=S(p,p)$ holds for the probability densities $p,q\in\mathcal{P}$. 
Then, the equality of H\"{o}lder's inequality should hold. Therefore, $p$ and $q$ are
linearly dependent, i.e., there exists a constant $c\in\Rbb$ such that $p=cq$ 
holds. For the probability densities, the constant $c$ should be $1$, and we obtain $p=q$. 
\end{proof}

\section{Bregman scores and H\"{o}lder scores}
\label{appendix:proof_Bregman_Holder}

\begin{proof}
 [proof of Theorem~\ref{theorem:Bregman-Holder}]
 We prove the first case. 
 Suppose that there exists a strictly monotone increasing function $\xi$ such that 
\begin{align}
 \label{eqn:Bregman-Holder-equiv}
& \phantom{\Longrightarrow}
 -G(g)-\int{}G_g^*(\omega)(f(\omega)-g(\omega))dm(\omega)
 = -\xi(-\phi(\<fg^\gamma\>/\<g^{1+\gamma}\>)\<g^{1+\gamma}\>)
\end{align}
 for all $f,g\in\mathcal{F}=L_{1+\gamma}^+$. 
 Here, the expression $-\xi(-\phi(\<fg^\gamma\>/\<g^{1+\gamma}\>)\<g^{1+\gamma}\>)$ 
 is used instead of $\xi(\phi(\<fg^\gamma\>/\<g^{1+\gamma}\>)\<g^{1+\gamma}\>)$
 for a simple expression of the potential function. 
 Substituting $f$ into $g$, we have
 $G(f)=\xi(\<f^{1+\gamma}\>)$. 
 For $\delta\in\Rbb$, the function $A(\delta)=\<|f+\delta{h}|^{1+\gamma}\>$ is
 differentiable at $\delta=0$  
 for all $f\in{L_{1+\gamma}^+}$ and all $h\in{L_{1+\gamma}}$, 
 and $A'(0)=(1+\gamma)\<f^\gamma{h}\>$ holds
 \cite[Chap.~8]{fabian01:_funct_analy_infin_dimen_geomet}. 
 In addition, the differentiability of the potential $G(f)$ is assumed. 
 We prove that the function $\xi$ is differentiable on $\Rbb_+^\circ$. 
 Let $a\in\Rbb$ be a real number with a small  absolute value, and 
 let us define $g=(1+a)f\in\mathcal{F}$ for a given $f\in\mathcal{F}$. 
 Then, $(1-\varepsilon)f+\varepsilon{g}=(1+a\varepsilon)f\in\mathcal{F}$ holds for
 $\varepsilon$ 
 with $|\varepsilon|<\delta$, where $\delta$ is a small positive constant. 
 Let the function $A(\varepsilon)$ be
 $A(\varepsilon)=G((1-\varepsilon)f+\varepsilon{g}) =
 \xi((1+a\varepsilon)^{1+\gamma}\<f^{1+\gamma}\>)$. 
 For all $f\in\mathcal{F}$, $A(\varepsilon)$ is differentiable at $\varepsilon=0$. 
 This implies that  $\xi(z)$ is differentiable for $z>0$. 
 
 We specify the expression of the function $\xi$.
 The (sub)gradient of $G(g)=\xi(\<g^{1+\gamma}\>)$ at $g\in\mathcal{F}$ is given as  
 \begin{align*}
  G_g^*(\omega)=(1+\gamma)\xi'(\<g^{1+\gamma}\>)g^\gamma(\omega). 
 \end{align*}
 Let $x=\<g^{1+\gamma}\>$ and $z=\<fg^\gamma\>/\<g^{1+\gamma}\>$ for $f,g\in\mathcal{F}$. 
 Then, $(x,z)$ can take any point in $\Rbb_+^\circ\times\Rbb_+^\circ$. 
 The equation \eqref{eqn:Bregman-Holder-equiv} is rewritten as 
 \begin{align*}
  \xi(x)+(1+\gamma)\xi'(x)(xz-x)=\xi(-\phi(z)x). 
 \end{align*}
 The continuous function $\phi$ satisfies
 the conditions in Definition~\ref{def:Holder-score}, i.e., 
 $\phi(1)=-1$ and $\phi(z)\geq-z^{1+\gamma}$ for $z\geq0$. 
 Hence, there exists a real number $z_0$ such that $0\leq{}z_0<1$ and $\phi(z_0)=0$. 
 Substituting $z=z_0$, we obtain the differential equation of $\xi(x)$, 
 \begin{align*}
  \xi(x)+(1+\gamma)(z_0-1)x\xi'(x)=\xi(0). 
 \end{align*} 
 The solution is given as 
 \begin{align*}
  \xi(x)=\xi(0)+c x^{1/((1+\gamma)(1-z_0))}, 
 \end{align*}
 where $c$ is a positive constant. 
 For $\kappa=1/(1-z_0)\geq1$, we have  
 $G(f)=\<f^{1+\gamma}\>^{\kappa/(1+\gamma)}$ up to an affine transformation with a
 positive factor. 
 Note that $\<f^{1+\gamma}\>^{\kappa/(1+\gamma)}$ with $\gamma>0$ and $\kappa\geq1$ is
 convex on $\mathcal{F}$ and strictly convex on $\mathcal{P}$. 

 Let us consider the second case. 
 Suppose that the potential function $G(f)=\<f^{1+\gamma}\>^{\kappa/(1+\gamma)}$ provides
 a separable Bregman divergence. Then, $\kappa$ should be $1+\gamma$. 

\end{proof}

\section{Affine invariant divergences}
\label{appendix:proof_affine_invariance}

Let $\Omega=\Rbb^d,\, \mathcal{B}$ be the Borel set of $\Omega$, and
$m:\mathcal{B}\rightarrow\Rbb_+$ be the Lebesgue measure on $(\Omega,\mathcal{B})$. 

\subsection{The functions $U$ and $V$}
\label{appendix:proof_UV}
We show the proof of Theorem  \ref{theorem:affine-invariance_UV}. 
Let us consider a necessary condition that the function 
\eqref{eqn:special-score} provides a composite score. 
\begin{lemma}
 \label{lemma:UV-relation}
 Under Assumption \ref{assumption:measure-space_F} and Assumption \ref{assumption:psi-U-V}, 
 the equality 
 \begin{align*}
  V(z)=c\!\int{}\!zU'(z)dz,\quad z>0
 \end{align*}
 holds, 
 where $c\in\Rbb$ is a non-zero constant. 
\end{lemma}

\begin{proof}
 [Proof of Lemma \ref{lemma:UV-relation}]
 Let $A$ and $B$ be disjoint measurable subsets of $(0,1)^d$ 
 such that $A\cup{B}=(0,1)^d$, and $m(A)$ and $m(B)$ are positive. 
 For $x=(x_1,x_2)\in\Rbb_+^\circ\times\Rbb_+^\circ$, 
 let us define the function class $f_x\in\mathcal{F}_0\subset\mathcal{F}$ as 
\begin{align*}
 f_x(\omega)=
 \begin{cases}
  x_1, & \omega\in{A},   \\
  x_2, & \omega\in{B},   \\
  \,0, & \text{otherwise}. 
 \end{cases}
\end{align*}
 For $x=(x_1,x_2)$ and $y=(y_1,y_2)$, we have 
\begin{align*}
 \<f_x{}U(f_y)\> &= x_1{}U(y_1)m(A)+x_2{}U(y_2)m(B),\\
 \<V(f_y)\>  &=  V(y_1)m(A)+V(y_2)m(B). 
\end{align*} 
 Since $S$ is a composite score, the inequality
\begin{align*}
& \phantom{\geq}
 \psi(x_1{}U(y_1)m(A)+x_2{}U(y_2)m(B),V(y_1)m(A)+V(y_2)m(B))\\
& \geq 
 \psi(x_1{}U(x_1)m(A)+x_2{}U(x_2)m(B),V(x_1)m(A)+V(x_2)m(B)) 
\end{align*}
 holds for $x_1,x_2,y_1,y_2>0$. 
 Hence, we have 
\begin{align*}
&\phantom{\Longleftrightarrow}
 \frac{\partial}{\partial{y_i}}
 \psi(x_1{}U(y_1)m(A)+x_2{}U(y_2)m(B),V(y_1)m(A)+V(y_2)m(B))\bigg|_{y=x}=0,\\
 &\Longleftrightarrow 
 \psi_1x_1U'(x_1)+\psi_2V'(x_1)=0,\quad
 \psi_1x_2U'(x_2)+\psi_2V'(x_2)=0, 
\end{align*}
for $i=1,2$, where $\psi_i$ is evaluated at $(\<f_xU(f_x)\>,\<V(f_x)\>)\in\Rbb^2$. 
 From Assumption~\ref{assumption:psi-U-V}(c), 
 the gradient vector of $\psi$ does not vanish. Therefore, the matrix 
\begin{align*}
 \begin{pmatrix}
  x_1U'(x_1) & V'(x_1)\\
  x_2U'(x_2) & V'(x_2)
 \end{pmatrix}
\end{align*}
is not invertible for all $x_1,x_2>0$. 
Thus, the equality 
\begin{align*}
 x_1U'(x_1)V'(x_2)-x_2U'(x_2)V'(x_1)=0
\end{align*}
 should hold for all $x_1,x_2>0$. 
 Since $U$ is not a constant function on $\Rbb_+^\circ$, there exists $x_2>0$ such that
 $U'(x_2)\neq0$. Hence, 
 we obtain the equalities, 
\begin{align*}
 V'(z)=c zU'(z)\ \ \text{and}\ \ V(z)=c\int{}zU'(z)dz,\quad z>0, 
\end{align*}
 with a non-zero constant $c$. 
\end{proof}

Below, we present the proof of Theorem~\ref{theorem:affine-invariance_UV}. 
\begin{proof}
 [Proof of Theorem~\ref{theorem:affine-invariance_UV}]
 We assume $\Omega=\Rbb$. Extension to the multi-dimensional case is
 straightforward. 
 For a positive real number $\sigma$, 
 let us consider the affine transformation $\omega\mapsto\sigma{\omega}$ for $\omega\in\Rbb$. 
 This action induces the transformation of the probability density, 
 $p(\omega)\mapsto{}p_\sigma(\omega)=\sigma{}p(\sigma{}\omega)$. 
 A simple calculation yields that the divergence $D(p_\sigma,q_\sigma)$
 is given as 
 \begin{align*}
 D(p_\sigma,q_\sigma)
 =
  \psi(\<pU(\sigma{}q)\>,\<V(\sigma{q})/\sigma\>)
 -\psi(\<pU(\sigma{}p)\>,\<V(\sigma{p})/\sigma\>). 
 \end{align*}
Let us define the function set $\mathcal{V}$ as 
\begin{align*}
 \mathcal{V}=\left\{v\in{}L_0\,\big|\,v(\omega)=0\ \text{for all}\
 \omega\not\in(0,1),\  \<v\>=0, \,
 \text{and}\ \|v\|_\infty<1 \right\}. 
\end{align*}
Let $u(\omega)$ be the probability density of the uniform distribution on the interval $(0,1)$,
i.e., $u(\omega)$ equals $1$ on $(0,1)$ and $0$ otherwise. 
For $v\in\mathcal{V}$ and $\varepsilon$ with $|\varepsilon|<1$, 
the function $p=u+\varepsilon{v}\in\mathcal{F}_0$ is also a probability density. 
Let $q(\omega)$ be a probability density in $\mathcal{F}_0$. 
We see that 
$D((u+\varepsilon{v})_\sigma,q_\sigma)$ is second order differentiable with respect
to $\sigma$ and $\varepsilon$ in the vicinity of $(\sigma,\varepsilon)=(1,0)$. 
This is confirmed by the dominating convergence theorem. 
Indeed, around $(\sigma,\varepsilon)=(1,0)$, the functions, 
$(u+\varepsilon{v})U(\sigma{q}), V(\sigma{q})/\sigma,
 (u+\varepsilon{v})U(\sigma(u+\varepsilon{v}))$ 
and $V(\sigma(u+\varepsilon{v}))/\sigma$,
and those derivatives are all bounded on the interval $(0,1)$, and they take zero on the
outside of the interval $(0,1)$.  
The scale function $h(\sigma)$ is differentiable around $\sigma=1$
because of 
the differentiability of $D({(u+\varepsilon{v})}_\sigma,q_\sigma)$ and 
the equality $h(\sigma)=D(u+\varepsilon{v},q)/D({(u+\varepsilon{v})}_\sigma,q_\sigma)$. 
The affine invariance of the divergence yields the equality
\begin{align}
 \label{eqn:appendix.affine-invariant-div}
 \frac{\partial}{\partial\sigma}h(\sigma)D((u+\varepsilon{v})_\sigma,q_\sigma)=0
\end{align}
for all $v\in\mathcal{V}$ and arbitrary $\varepsilon$ with $|\varepsilon|<1$. 
Therefore, we have 
\begin{align*}
 \frac{\partial^2}{\partial\varepsilon\partial\sigma}h(\sigma)D((u+\varepsilon{v})_\sigma,q_\sigma)
 \bigg|_{\substack{\sigma=1\\\varepsilon=0}} =0. 
\end{align*}
for all $v\in\mathcal{V}$. 
The equality above produces 
\begin{align*}
 \int_\Omega\{c_1 U(q(\omega))+c_2{}U'(q(\omega))q(\omega)\}v(\omega)dm(\omega)=0, 
\end{align*}
for all $v\in\mathcal{V}$, where $c_1$ and $c_2$ are some constants. 
Therefore, there exists another constant $c_3$ such that the equality
\begin{align*}
 c_1 U(q(\omega))+c_2{}U'(q(\omega))q(\omega) =c_3
\end{align*}
should hold for all $\omega\in(0,1)$. 
Here, $q$ is an arbitrary probability density satisfying the inequality 
$0<a<q(\omega)<b$
on the support $(0,1)$. 
Since $a$ and $b$ can take arbitrary positive numbers such that $0<a<1<b$, 
the function $U$ should satisfy the differential equation 
\begin{align*}
 c_1 U(z)+c_2{}U'(z)z =c_3, \quad z>0. 
\end{align*}
Up to a constant factor, the solution is given as 
$U(z)=z^\gamma+c$ or $U(z)=-\log{z}+c$. 
From Lemma~\ref{lemma:UV-relation}, we conclude that the corresponding $V$ is
 $V(z)=z^{1+\gamma}$ for $U(z)=z^\gamma+c$, and $V(z)=z$ for $U(z)=-\log{z}+c$ up to a constant factor. 
Since the equality $\lim_{z\searrow0}V(z)=V(0)=0$ and the existence of
$\lim_{z\searrow0}V'(z)$ are assumed in Assumption~\ref{assumption:psi-U-V}~(b), the real
number $\gamma$ of $U(z)=z^\gamma+c$ 
should be positive. 
\end{proof}

\subsection{The proof of Theorem~\ref{theorem:psi_form}}
\label{appendix:subsec:proof-affine-invariant-psi}

\subsubsection{proof of the case \ref{case-1}}
Let the functions $U$ and $V$ in \eqref{eqn:special-score} be $U(z)=-\log{z}+c$ and $V(z)=z$. 
\begin{proof}
 [Proof of the case \ref{case-1} in Theorem  \ref{theorem:psi_form}]
 For $U(z)=-\log{z}+c$ and $V(z)=z$, the composite score is given as 
 $S(f,g)=\psi(\<{-f\log{g}+cf}\>,\<{g}\>)$. 
 For the probability densities $p,q\in\mathcal{P}\subset\mathcal{F}$, 
 the composite score satisfies the inequality 
 $\psi(c-\<p\log{q}\>,1)\geq \psi(c-\<p\log{p}\>,1)$. 
 Hence, the function $\psi(\cdot,1)$ should be strictly increasing, 
 since $-\<p\log{q}\>\geq-\<p\log{p}\>$ holds for any distinct $p,q$ in $\mathcal{P}$. 
 Therefore, $S(f,g)$ is equivalent in probability with the {\rm KL} score. 
\end{proof}

\subsubsection{proof of the case \ref{case-2}}

We prepare some lemmas. 
\begin{lemma}
 \label{lemma:psi-phi}
 Suppose $U(z)=z^\gamma+c$ and $V(z)=z^{1+\gamma}$. 
 Under the assumption in Theorem~\ref{theorem:psi_form}, 
 there exists a function $\phi:\Rbb\rightarrow\Rbb$ and $s\in\Rbb$ such that 
 the function $\psi(x,y)$ in \eqref{eqn:special-score} 
 is represented as $\psi(x,y)=\phi((x-c)/y)y^s$  
 up to a monotone transformation. 
\end{lemma}

\begin{proof}
 [Proof of Lemma \ref{lemma:psi-phi}]
 For $U(z)=z^\gamma+c,\, V(z)=z^{1+\gamma}$, we have 
 $S(p,q)=\psi(\<{pq^\gamma}\>+c,\<q^{1+\gamma}\>)$ for $p,q\in\mathcal{P}$. 
 By replacing $\psi(x+c,y)$ with $\psi(x,y)$, the composite score on $\mathcal{P}$ 
 is represented as $S(p,q)=\psi(\<{pq^\gamma}\>,\<q^{1+\gamma}\>)$. 
 For $p,q\in{}\mathcal{P}\subset{}L_{1+\gamma}^+$, the integrals $\<p^{1+\gamma}\>$ and
 $\<pq^\gamma\>$ are finite. 
 Let us consider the affine transformation $\omega\mapsto{}\sigma{\omega}$ 
 on $\Omega=\Rbb$, where $\sigma>0$. 
 In the same way as the derivation of \eqref{eqn:appendix.affine-invariant-div} in the proof of Theorem
 \ref{theorem:affine-invariance_UV}, we have 
\begin{align*}
 \frac{\partial}{\partial\sigma} 
 h(\sigma)\left\{\psi(\sigma^\gamma{\<pq^\gamma\>},\sigma^\gamma{\<q^{1+\gamma}\>})
 -\psi(\sigma^\gamma{\<p^{1+\gamma}\>},\sigma^\gamma{\<p^{1+\gamma}\>})\right\}
 \bigg|_{\sigma=1}=0, 
\end{align*}
 where $h(\sigma)$ is the scale function. 
Let us define
$x={\<pq^\gamma\>}, y={\<q^{1+\gamma}\>}, z={\<p^{1+\gamma}\>}$, and 
$s=-\frac{d}{d\sigma}\log{h(\sigma)}\big|_{\sigma=1}\in\Rbb$. 
Then, we have 
 \begin{align*}
  -s\psi(x,y)+x\psi_1(x,y)+y\psi_2(x,y)  =  -s\psi(z,z)+z\psi_1(z,z)+z\psi_2(z,z). 
 \end{align*}
 Note that $(x,y,z)$ are independent variables in an open subset of $\Rbb^3$. 
 One can prove this fact by using the implicit function theorem. 
 Thus, the left side of the above equation should be a constant for any $(x,y)$ in an open subset of $\Rbb^2$, 
 since the right side is independent of $(x,y)$. 
 Hence, there exists a real number $b\in\Rbb$ such that 
\begin{align*}
 -s\psi(x,y)+x\psi_1(x,y)+y\psi_2(x,y)=b.  
\end{align*}
 The general solution of this partial differential equation is found from Euler's equation 
 \cite{dawid12:_proper_local_scorin_rules_discr_sampl_spaces}. 
 Here, we solve the above PDE by using the variable change. 
 For the polar coordinate system $(r,\theta)$ of $\Rbb^2$ with $x=r\cos\theta$ and $y=r\sin\theta$, 
 the above PDE is expressed as 
 \begin{align}
  -s\bar{\psi}(r,\theta)+r\frac{\partial}{\partial{r}}\bar{\psi}(r,\theta)=b, 
  \label{eqn:affine_invariant_equation_f_C}
 \end{align}
 where $\bar{\psi}(r,\theta)=\psi(r\cos\theta,r\sin\theta)$. 
 All solutions are given by 
 \begin{align*}
  \bar{\psi}(r,\theta)=
  \bar{\phi}(\theta)r^{s}+ 
  \begin{cases} 
   -b/s,     & s\neq 0,\\
   b\log{r}, & s=0,
  \end{cases}
 \end{align*}
 where $\bar{\phi}(\theta)$ is a function of $\theta$. 
 In the $(x,y)$-coordinate system, there exists a function $\phi$ such that 
 \begin{align*}
  \psi(x,y)=\phi(x/y)y^s
  +
  \begin{cases}
   c_1,        & s\neq0,\\
   c_0\log{y}, & s=0,
  \end{cases}
 \end{align*}
 where $c_0,c_1\in\Rbb$. Without loss of generality we set $c_1=0$. 
 For $s=0$, we have $e^{\psi(x,y)}=e^{\phi(x/y)}y^{c_0}$. 
 Hence $\psi(x,y)$ or $e^{\psi(x,y)}$ can be expressed as the form of 
 $\phi(x/y)y^s$ with $s\in\Rbb$. 
\end{proof}

 Let $U(z)=z^\gamma+c$ and $V(z)=z^{1+\gamma}$ with $\gamma>0$ and $c\in\Rbb$. 
 Then, Lemma~\ref{lemma:psi-phi} ensures that
 for $f\in\mathcal{P}$ and $g\in\mathcal{F}$, 
 the affine invariant composite score is of the form 
 \begin{align}
  H(f,g)=\phi\left(\frac{\<fg^\gamma\>}{\<g^{1+\gamma}\>}\right)\<g^{1+\gamma}\>^s. 
  \label{eqn:general-expression-holder}
 \end{align}
 with $s\in\Rbb$ up to a monotone transformation. 
 The sign of the parameter $s$ is determined by the following lemma. 

\begin{lemma}
 \label{lemma:holder-equivalence}
 For $\gamma>0$, let $\mathcal{F}=L_{1+\gamma}^+$ and
 $\mathcal{P}=\{p\in\mathcal{F}\,|\,\<p\>=1\}$. 
 Suppose that $H(f,g)$ in 
 \eqref{eqn:general-expression-holder} 
 is the composite score on $\mathcal{P}\times\mathcal{F}$, i.e., 
 $H(f,g)\geq{}H(f,f)$ for all $(f,g)\in\mathcal{P}\times\mathcal{F}$, and 
 $H(p,q)=H(p,p)$ for $(p,q)\in\mathcal{P}\times\mathcal{P}$ 
 implies $p=q$. 
 Then, $s>0>\phi(1)$ and $\phi(z)\geq\phi(1)z^{(1+\gamma)s}$ for $z\geq0$ hold. 
\end{lemma}

\begin{proof}
 [Proof of Lemma \ref{lemma:holder-equivalence}]
 Remember that the H\"{o}lder's inequality is represented as 
\begin{align}
 \<fg^\gamma\>\leq  
 \<f^{1+\gamma}\>^{1/(1+\gamma)}
 \<g^{1+\gamma}\>^{\gamma/(1+\gamma)},\quad 
 f,g\in{}\mathcal{F}=L_{1+\gamma}^+. 
 \label{eqn:Holder-inequality}
\end{align}
 The equality holds if and only if $f$ and $g$ are linearly dependent.

 First of all, we prove $\phi(1)\neq0$ and $s\neq0$. 
 Suppose that $\phi(1)=0$ holds. 
 Then, the equality
 \begin{align*}
  H(p,q)-H(p,p)=\phi\left(\frac{\<pq^\gamma\>}{\<q^{1+\gamma}\>}\right)\<q^{1+\gamma}\>^s=0
 \end{align*}
 holds for $p,q\in\mathcal{P}$ if and only if $p=q$. 
 Let $q$ be the probability density of
 the uniform distribution on $(0,1)^d\subset\Omega=\Rbb^d$. 
 Then, arbitrary probability density $p$ whose
 support is included in $(0,1)^d$ satisfies $H(p,q)-H(p,p)=\phi(1)=0$. 
 This contradicts the assumption that $H$ is the composite score. 
 Therefore, $\phi(1)\neq0$ holds. 
 Suppose $s=0$. Then, the equality
 \begin{align*}
  H(p,q)-H(p,p)=\phi\left(\frac{\<pq^\gamma\>}{\<q^{1+\gamma}\>}\right)-\phi(1)=0
 \end{align*}
 holds for $p,q\in\mathcal{P}$ if and only if $p=q$. 
 In the same way as above, 
 setting $q$ as the probability density of the uniform distribution
 on $(0,1)^d$ yields the contradiction. Therefore, we obtain $s\neq0$.

 Next, we prove $\phi(0)\geq0>\phi(1)$. 
 Let $A$ and $B$ be disjoint subsets of $\Omega=\Rbb^d$, and suppose that 
 they have finite positive measures. 
 Let $p$ and $q$ be the probability densities of the uniform distribution 
 on $A$ and $B$, respectively. 
 Then, we have $\<p^{1+\gamma}\>=m(A)^{-\gamma}, \<q^{1+\gamma}\>=m(B)^{-\gamma}$
 and $\<pq^\gamma\>=0$. 
 For the composite score $H(p,q)$, the inequality 
 \begin{align*}
  H(p,q)-H(p,p)=\phi(0)m(B)^{-\gamma{s}}-\phi(1)m(A)^{-\gamma{s}}\geq0
 \end{align*}
 holds. For $\gamma>0$ and $s\neq0$, 
 $m(A)^{-\gamma{s}}$ and $m(B)^{-\gamma{s}}$ can take any positive real numbers
 independently. 
 Hence, the inequality $\phi(0)\geq0\geq\phi(1)$ should hold. 
 This result and $\phi(1)\neq0$ lead to $\phi(0)\geq0>\phi(1)$. 

 Let us consider the sign of $s$. Since $H$ is the composite score, the inequality 
 \begin{align*}
  H(f,g)-H(f,f)
  =
  \left\{
  \phi\left(\frac{\<fg^\gamma\>}{\<g^{1+\gamma}\>}\right)\frac{\<g^{1+\gamma}\>^s}{\<f^{1+\gamma}\>^s}
  -\phi(1)
  \right\}
  \<f^{1+\gamma}\>^s
  \geq0
 \end{align*}
 holds for all $f\in\mathcal{P}$ and $g\in\mathcal{F}$. 
 There exist $f\in\mathcal{P}$ and $g\in\mathcal{F}$ such that 
 \begin{align}
  \label{eqn:special-inequality}
  1=\frac{\<fg^\gamma\>}{\<g^{1+\gamma}\>}
  <\left(\frac{\<f^{1+\gamma}\>}{\<g^{1+\gamma}\>}\right)^{1/(1+\gamma)}
  <\frac{\<f^{1+\gamma}\>}{\<g^{1+\gamma}\>}
 \end{align}
 holds, i.e., the H\"{o}lder's inequality strictly holds with 
 $1=\<fg^\gamma\>/\<g^{1+\gamma}\>$. 
 For example, for linearly independent functions, $f\in\mathcal{P}$ and
 $g_0\in\mathcal{F}$, with $\<fg_0^\gamma\>\neq0$, 
 let $g$ be $g_0\<fg^\gamma_0\>/\<g_0^{1+\gamma}\>$. 
 For $f\in\mathcal{P},\,g\in\mathcal{F}$ satisfying \eqref{eqn:special-inequality}, we have the inequality
 \begin{align*}
  \phi\left(\frac{\<fg^\gamma\>}{\<g^{1+\gamma}\>}\right)\frac{\<g^{1+\gamma}\>^s}{\<f^{1+\gamma}\>^s}-\phi(1)
  =
  \phi(1)\left(\frac{\<g^{1+\gamma}\>^s}{\<f^{1+\gamma}\>^s}-1\right)\geq 0, 
 \end{align*}
 from the non-negativity of $H(f,g)-H(f,f)$ and positivity of $\<f^{1+\gamma}\>$. 
 From $0<\<g^{1+\gamma}\>/\<f^{1+\gamma}\><1,\,\phi(1)<0$ and $s\neq0$,  
 the inequality above holds only when $s>0$.

 Suppose that there exists $z_0>0$ such that $\phi(z_0)<\phi(1)z_0^{(1+\gamma)s}$ 
 holds. Choose $f\in\mathcal{P}$ and $g\in\mathcal{F}$ such that 
 \begin{align*}
  \left(\frac{\<fg^\gamma\>}{\<g^{1+\gamma}\>}\right)^{1+\gamma}
  =\frac{\<f^{1+\gamma}\>}{\<g^{1+\gamma}\>}
  =z_0^{1+\gamma}
 \end{align*}
 holds. This is possible by choosing, say, $g=f/z_0\in\mathcal{F}$ for some $f\in\mathcal{P}$. 
 For such $f$ and $g$, we have
 \begin{align*}
  H(f,g)-H(f,f)
  &=
  \phi(z_0)\<g^{1+\gamma}\>^s-\phi(1)\<f^{1+\gamma}\>^s\\
  &<
  \phi(1)z_0^{(1+\gamma)s}\<g^{1+\gamma}\>^s-\phi(1)\<f^{1+\gamma}\>^s\\
  &=\phi(1)
  \frac{\<f^{1+\gamma}\>^s}{\<g^{1+\gamma}\>^s}\<g^{1+\gamma}\>^s-\phi(1)\<f^{1+\gamma}\>^s\\
  &=0, 
 \end{align*}
 in which $\<g^{1+\gamma}\>>0$ is used. 
 This is the contradiction. 
 Therefore, the inequality
 $\phi(z)\geq\phi(1)z^{(1+\gamma)s}$ should hold for all $z>0$. 
 From $\phi(0)\geq0$ and $(1+\gamma)s>0$, eventually the inequality 
 $\phi(z)\geq\phi(1)z^{(1+\gamma)s}$ should hold for all $z\geq0$. 
\end{proof}

Finally, we prove the case \ref{case-2} of Theorem  \ref{theorem:psi_form}. 
\begin{proof}
[Proof of the case \ref{case-2} in Theorem  \ref{theorem:psi_form}]
From Lemma~\ref{lemma:psi-phi} and Lemma~\ref{lemma:holder-equivalence}, 
the affine invariant composite score is expressed as 
\begin{align*}
H(p,q)=\phi\left(\frac{\<pq^\gamma\>}{\<q^{1+\gamma}\>}\right)\<q^{1+\gamma}\>^s\quad
 \text{for}\ \ p,q\in\mathcal{P},
\end{align*}
with $\gamma>0$, where $\phi(z)\geq\phi(1)z^{(1+\gamma)s}$ for $z\geq0$ and
$s>0>\phi(1)$ hold. 
 The transformation using the strictly increasing function 
 $\xi(H)=|H/\phi(1)|^{1/s}\mathrm{sign}(H)$  
 ensures that the composite score $H$ is equivalent in probability with 
 the H\"{o}lder score with $\gamma>0$. 
 The inequality $\phi(z)\geq\phi(1)z^{(1+\gamma)s}$ with $\phi(1)<0$ is
 transformed into $\phi(z)\geq-z^{1+\gamma}$. 
\end{proof}

\section{Redescending property}
\label{appendix:proof_redescending_property}
For a differentiable real-valued function $f(\theta)$ of $\theta\in\Rbb^k$, 
let $\frac{\partial{f}}{\partial\theta}$ be 
the gradient column vector of $f(\theta)$. 

\begin{proof}
 [Proof of Theorem~\ref{theorem:redescending_prop}]
 Let us define 
 $p_\varepsilon=(1-\varepsilon)p_{\theta^*}+\varepsilon{}\delta_z(x)=p_{\theta^*}+\varepsilon{}(\delta_z(x)-p_{\theta^*}(x))$, 
 and $r_z(x)$ be $r_z(x)=\delta_z(x)-p_{\theta^*}(x)$. 
By using the implicit function theorem to the $\Rbb^k$-valued function 
\begin{align*}
 (\theta,\varepsilon)
 \ \longmapsto\ 
 \frac{\partial}{\partial\theta}
 \left\{
 \phi\left(\frac{\<p_{\varepsilon}p_\theta^\gamma\>}{\<p_\theta^{1+\gamma}\>}\right)\<p_\theta^{1+\gamma}\>
 \right\}
\end{align*}
 around $(\theta,\varepsilon)=(\theta^*,0)$, we obtain 
\begin{align}
 \label{eqn:appendix:influence-function}
 \mathrm{IF}(z,\theta^*,S)=
 -I^{-1}
 \frac{\partial}{\partial\theta}\left\{
 \phi'\bigg(\frac{\<p_{\theta^*}p_\theta^\gamma\>}{\<p_\theta^{1+\gamma}\>}\bigg)
 \<r_zp_\theta^\gamma\>
\right\}
 \bigg|_{\theta=\theta^*}. 
\end{align}
 Hence, the estimator has the redescending property if and only if 
 \begin{align*}
  \lim_{\|z\|\rightarrow\infty}
 \frac{\partial}{\partial\theta}
 \left\{
  \phi'\bigg(\frac{\<p_{\theta^*}p_\theta^\gamma\>}{\<p_\theta^{1+\gamma}\>}\bigg) \<r_zp_\theta^\gamma\>
 \right\}
 \bigg|_{\theta=\theta^*}=0
 \end{align*}
 holds for any $\theta^*\in\Theta$. 
 From the assumption on $\phi$, we have $\phi'(1)=-1-\gamma$. 
 A calculation using $\phi(1)=-1$ and $\phi'(1)=-1-\gamma$ yields that the derivative
 in the above is given as 
\begin{align*}
 \frac{\partial}{\partial\theta}
  \phi'\bigg(\frac{\<p_{\theta^*}p_\theta^\gamma\>}{\<p_\theta^{1+\gamma}\>}\bigg) \<r_zp_\theta^\gamma\>
 \bigg|_{\theta=\theta^*}
&=
 -\phi''(1)\frac{\<r_zp_{\theta^*}^\gamma\>}{\<p_{\theta^*}^{1+\gamma}\>}\int{}p_{\theta^*}(x)^{1+\gamma}s_{\theta^*}(x)dm(x)\\
&\phantom{=} -\gamma(1+\gamma)\int{}r_z(x)p_{\theta^*}(x)^\gamma{}s_{\theta^*}(x)dm(x), 
\end{align*}
 in which the interchangeability of the integral and differential is used. 
 From the assumption, the limiting of $\|z\|\rightarrow\infty$ leads to 
\begin{align*}
&\phantom{=} \lim_{\|z\|\rightarrow\infty}
\frac{\partial}{\partial\theta} 
 \left\{
 \phi'\bigg(\frac{\<p_{\theta^*}p_\theta^\gamma\>}{\<p_\theta^{1+\gamma}\>}\bigg) \<r_zp_\theta^\gamma\>
 \right\}
 \bigg|_{\theta=\theta^*}\\
&=~
 (\phi''(1)+\gamma(1+\gamma))\int{}p_{\theta^*}(x)^{1+\gamma}s_{\theta^*}(x)dm(x). 
\end{align*} 
 The expression above vanishes for all $\theta^*$ if and only if
 the equality $\phi''(1)=-\gamma(1+\gamma)$ holds. 

 The asymptotic variance of the estimator is determined from the influence function. 
 Some calculation shows that 
 H\"{o}lder score affects the influence function via $\phi''(1)$. 
 Hence, the optimum score estimators using H\"{o}lder scores with the same $\phi''(1)$ 
 have the same asymptotic variance. 
\end{proof}

\bibliographystyle{plain}

\begin{thebibliography}{10}

\bibitem{abernethy12:_charac_scorin_rules_linear_proper}
J.~D. Abernethy and R.~M. Frongillo.
\newblock A characterization of scoring rules for linear properties.
\newblock {\em Journal of Machine Learning Research - Proceedings Track},
  23:27.1--27.13, 2012.

\bibitem{banerjee05:_clust_bregm_diver}
A.~Banerjee, S.~Merugu, I.~S. Dhillon, and J.~Ghosh.
\newblock Clustering with {B}regman divergences.
\newblock {\em J. Mach. Learn. Res.}, 6:1705--1749, December 2005.

\bibitem{a.98:_robus_effic_estim_minim_densit_power_diver}
A.~Basu, I.~R. Harris, N.~L. Hjort, and M.~C. Jones.
\newblock Robust and efficient estimation by minimising a density power
  divergence.
\newblock {\em Biometrika}, 85(3):549--559, 1998.

\bibitem{basu10:_statis_infer}
A.~Basu, H.~Shioya, and C.~Park.
\newblock {\em Statistical Inference: The Minimum Distance Approach}.
\newblock Monographs on Statistics and Applied Probability. Taylor \& Francis,
  2010.

\bibitem{berger85:_statis_decis_theor_bayes_analy}
J.~O. Berger.
\newblock {\em Statistical Decision Theory and Bayesian Analysis}.
\newblock Springer Series in Statistics. Springer, 1985.

\bibitem{borwein05:_techn}
J.~M. Borwein and Q.~Q.~J. Zhu.
\newblock {\em Techniques of variational analysis}.
\newblock CMS books in mathematics. Springer Science+Business Media,
  Incorporated, 2005.

\bibitem{bregman67:_relax_method_of_findin_commonc}
L.~M. Bregman.
\newblock The relaxation method of finding the common point of convex sets and
  its application to the solution of problems in convex prog ramming.
\newblock {\em USSR Computational Mathematics and Mathematical Physics},
  7:200--217, 1967.

\bibitem{bremnes04:_probab_forec_precip_terms_quant}
B.~J. Bremnes.
\newblock Probabilistic forecasts of precipitation in terms of quantiles using
  nwp model output.
\newblock {\em Monthly Weather Review}, 132:338--347, 2004.

\bibitem{brier50:_verif}
G.~W. Brier.
\newblock Verification of forecasts expressed in terms of probability.
\newblock {\em Monthly Weather Rev.}, 78:1--3, 1950.

\bibitem{cichocki10:_famil_alpha_beta_gamma_diver}
A.~Cichocki and S.~Amari.
\newblock Families of alpha- beta- and gamma- divergences: Flexible and robust
  measures of similarities.
\newblock {\em Entropy}, 12(6):1532--1568, 2010.

\bibitem{Collins_etal00}
M.~Collins, R.~E. Schapire, and Y.~Singer.
\newblock Logistic regression, adaboost and {Bregman} distances.
\newblock In {\em Proceedings of the Thirteenth Annual Conference on
  Computational Learning Theory}, pages 158--169, 2000.

\bibitem{dawid98:_coher_measur_discr_uncer_depen}
A.~P. Dawid.
\newblock Coherent measures of discrepancy, uncertainty and dependence, with
  applications to bayesian predictive experimental design.
\newblock Technical report, University College London, Dept. of Statistical
  Science, 1998.

\bibitem{dawid07}
A.~P. Dawid.
\newblock The geometry of proper scoring rules.
\newblock {\em Annals of the Institute of Statistical Mathematics},
  59(1):77--93, 2007.

\bibitem{dawid12:_proper_local_scorin_rules_discr_sampl_spaces}
A.~P. Dawid, S.~Lauritzen, and M.~Parry.
\newblock Proper local scoring rules on discrete sample spaces.
\newblock {\em Annals of Statistics}, 40:593--608, 2012.

\bibitem{duffie97:_overv_value_risk}
D.~Duffie and J.~Pan.
\newblock An overview of value at risk.
\newblock {\em Journal of Derivatives}, 4:7–49, 1997.

\bibitem{eguchi11:_projec_power_entrop_maxim_tsall_entrop_distr}
S.~Eguchi, O.~Komori, and S.~Kato.
\newblock Projective power entropy and maximum {Tsallis} entropy distributions.
\newblock {\em Entropy}, 13:1746--1764, 2011.

\bibitem{ehm12:_local}
W.~Ehm and T.~Gneiting.
\newblock Local proper scoring rules of order two.
\newblock {\em Annals of Statistics}, 40:609--637, 2012.

\bibitem{fabian01:_funct_analy_infin_dimen_geomet}
M.~Fabian, P.~Habala, P.~H{\'{a}}jek, V.~Montesinos Santalucia, J.~Pelant, and
  V.~Zizler.
\newblock {\em Functional Analysis and Infinite-Dimensional Geometry}.
\newblock CMS Books in Mathematics. Springer, 2001.

\bibitem{fujisawa08:_robus}
H.~Fujisawa and S.~Eguchi.
\newblock Robust parameter estimation with a small bias against heavy
  contamination.
\newblock {\em J. Multivar. Anal.}, 99(9):2053--2081, 2008.

\bibitem{gneiting07:_stric_proper_scorin_rules_predic_estim}
T.~Gneiting and A.~E. Raftery.
\newblock Strictly proper scoring rules, prediction, and estimation.
\newblock {\em Journal of the American Statistical Association}, 102:359--378,
  2007.

\bibitem{good71:_commen_measur_infor_uncer_r}
I.~J. Good.
\newblock Comment on "measuring information and uncertainty," by {R. J.
  Buehler}.
\newblock In V.~P. Godambe and D.~A. Sprott, editors, {\em Foundations of
  Statistical Inference}, page 337–339, Toronto: Holt, Rinehart and Winston,
  1971.

\bibitem{gr03:_game_theor_maxim_entrop_minim}
P.~D. Gr{\"{u}}nwald and A.~P. Dawid.
\newblock Game theory, maximum entropy, minimum discrepancy, and robust
  bayesian decision theory.
\newblock {\em Annals of Statistics}, 32:1367--1433, 2003.

\bibitem{Hampel_etal86}
F.~R. Hampel, P.~J. Rousseeuw, E.~M. Ronchetti, and W.~A. Stahel.
\newblock {\em Robust Statistics. The Approach based on Influence Functions}.
\newblock John Wiley and Sons, Inc., 1986.

\bibitem{hendrickson71:_proper_scores_probab_forec}
A.~D. Hendrickson and R.~J. Buehler.
\newblock Proper scores for probability forecasters.
\newblock {\em The Annals of Mathematical Statistics}, 42:1916–1921, 1971.

\bibitem{huber64:_robus}
P.~J. Huber.
\newblock Robust estimation of a location parameter.
\newblock {\em Annals of Mathematical Statistics}, 35(1):73--101, 1964.

\bibitem{jones01:_compar}
M.~C. Jones, N.~L. Hjort, I.~R. Harris, and A.~Basu.
\newblock A comparison of related density-based minimum divergence estimators.
\newblock {\em Biometrika}, 88(3):865--873, 2001.

\bibitem{maronna06:_robus_statis}
R.~Maronna, R.D. Martin, and V.~Yohai.
\newblock {\em Robust Statistics: Theory and Methods}.
\newblock Wiley, 2006.

\bibitem{murata04:_infor_geomet_u_boost_bregm_diver}
N.~Murata, T.~Takenouchi, T.~Kanamori, and S.~Eguchi.
\newblock Information geometry of {$U$-Boost} and {Bregman} divergence.
\newblock {\em Neural Computation}, 16(7):1437--1481, 2004.

\bibitem{parry12:_proper_local_scorin_rules}
M.~Parry, A.~P. Dawid, and S.~Lauritzen.
\newblock Proper local scoring rules.
\newblock {\em Annals of Statistics}, 40:561--592, 2012.

\bibitem{tsallis88:_possib_boltz_gibbs}
C.~Tsallis.
\newblock Possible generalization of {B}oltzmann-{G}ibbs statistics.
\newblock {\em Journal of Statistical Physics}, 52(1-2):479--487, 1988.

\bibitem{tsuda05:_matrix_expon_gradien_updat_learn_bregm_projec}
K.~Tsuda, G.~R{\"{a}}tsch, and M.~K. Warmuth.
\newblock Matrix exponentiated gradient updates for on-line learning and
  {B}regman projection.
\newblock {\em J. Mach. Learn. Res.}, 6:995--1018, 2005.

\bibitem{vaart00:_asymp_statis}
A.~W. van~der Vaart.
\newblock {\em Asymptotic Statistics}.
\newblock Cambridge Series in Statistical and Probabilistic Mathematics.
  Cambridge University Press, 2000.

\end{thebibliography}

\end{document}